\documentclass[
10pt, 
a4paper, 
oneside, 
headinclude,footinclude, 
BCOR=5mm, 
abstract=on 
]{scrartcl}

\usepackage[
nochapters, 
beramono, 
eulermath,
eulerchapternumbers,
pdfspacing, 
dottedtoc 
]{classicthesis} 
\usepackage{arsclassica} 

\usepackage{fullpage}

\usepackage{tikz-cd}

\usepackage{amsmath,amssymb,amsthm} 

\usepackage[all]{xy}
\usepackage{tensor}

\usepackage[nobysame,alphabetic,initials,msc-links,lite,backrefs]{amsrefs}

\usepackage{mathscinet}

\makeatletter
\renewcommand{\BibLabel}{%
    \Hy@raisedlink{\hyper@anchorstart{cite.\CurrentBib}\hyper@anchorend}%
    [\thebib]%
}
\makeatother

\DefineSimpleKey{bib}{how}
\renewcommand{\PrintDOI}[1]{%
  \href{http://dx.doi.org/#1}{\texttt{DOI:#1}}%
}
\renewcommand{\eprint}[1]{#1}
\BibSpec{book}{%
    +{}  {\PrintPrimary}                {transition}
    +{.} { \PrintDate}                  {date}
    +{.} { \textit}                     {title}
    +{.} { }                            {part}
    +{:} { \textit}                     {subtitle}
    +{,} { \PrintEdition}               {edition}
    +{}  { \PrintEditorsB}              {editor}
    +{,} { \PrintTranslatorsC}          {translator}
    +{,} { \PrintContributions}         {contribution}
    +{,} { }                            {series}
    +{,} { \voltext}                    {volume}
    +{,} { }                            {publisher}
    +{,} { }                            {organization}
    +{,} { }                            {address}
    +{,} { }                            {status}
    +{,} { \PrintDOI}                   {doi}
    +{,} { \PrintISBNs}                 {isbn}
    +{}  { \parenthesize}               {language}
    +{}  { \PrintTranslation}           {translation}
    +{;} { \PrintReprint}               {reprint}
    +{.} { }                            {note}
    +{.} {}                             {transition}
    +{}  {\SentenceSpace \PrintReviews} {review}
}
\BibSpec{article}{%
    +{}  {\PrintAuthors}                {author}
    +{,} { \textit}                     {title}
    +{.} { }                            {part}
    +{:} { \textit}                     {subtitle}
    +{,} { \PrintContributions}         {contribution}
    +{.} { \PrintPartials}              {partial}
    +{,} { }                            {journal}
    +{}  { \textbf}                     {volume}
    +{}  { \PrintDatePV}                {date}
    +{,} { \issuetext}                  {number}
    +{,} { \eprintpages}                {pages}
    +{,} { }                            {status}
    +{,} { \PrintDOI}                   {doi}
    +{,} { \eprint}        {eprint}
    +{}  { \parenthesize}               {language}
    +{}  { \PrintTranslation}           {translation}
    +{;} { \PrintReprint}               {reprint}
    +{.} { }                            {note}
    +{.} {}                             {transition}
    +{}  {\SentenceSpace \PrintReviews} {review}
}
\BibSpec{collection.article}{%
    +{}  {\PrintAuthors}                {author}
    +{,} { \textit}                     {title}
    +{.} { }                            {part}
    +{:} { \textit}                     {subtitle}
    +{,} { \PrintContributions}         {contribution}
    +{,} { \PrintConference}            {conference}
    +{}  {\PrintBook}                   {book}
    +{,} { }                            {booktitle}
    +{,} { \PrintDateB}                 {date}
    +{,} { pp.~}                        {pages}
    +{,} { }                            {publisher}
    +{,} { }                            {organization}
    +{,} { }                            {address}
    +{,} { }                            {status}
    +{,} { \PrintDOI}                   {doi}
    +{,} { \eprint}        {eprint}
    +{}  { \parenthesize}               {language}
    +{}  { \PrintTranslation}           {translation}
    +{;} { \PrintReprint}               {reprint}
    +{.} { }                            {note}
    +{.} {}                             {transition}
    +{}  {\SentenceSpace \PrintReviews} {review}
}
\BibSpec{misc}{%
  +{}{\PrintAuthors}  {author}
  +{,}{ \textit}      {title}
  +{.}{ }             {how}
  +{}{ \parenthesize} {date}
  +{,} { available at \eprint}        {eprint}
  +{,}{ available at \url}{url}
  +{,}{ }             {note}
  +{.}{}              {transition}
}

\numberwithin{equation}{section}

\newtheorem{theorem}{Theorem}[section]
\newtheorem{lemma}[theorem]{Lemma}
\newtheorem{proposition}[theorem]{Proposition}
\newtheorem{corollary}[theorem]{Corollary}

\theoremstyle{definition}

\theoremstyle{remark}

\newtheorem{remark}[theorem]{Remark}

\newtheorem{example}[theorem]{Example}

\newcommand{\Z}{{\mathbb{Z}}}
\newcommand{\R}{{\mathbb{R}}}
\newcommand{\C}{{\mathbb{C}}}
\newcommand{\T}{{\mathbb{T}}}

\newcommand{\cC}{\mathcal{C}}

\newcommand{\cA}{\mathcal{A}}
\newcommand{\cU}{\mathcal{U}}

\newcommand{\cM}{\mathcal{M}}
\newcommand{\GL}{\mathrm{GL}}
\newcommand{\ptimes}{\mathbin{\hat{\otimes}}}

\newcommand{\ubE}{\underline{E}}
\newcommand{\Ban}{\mathrm{Ban}}
\newcommand{\shrpbin}{\mathbin{\#}}

\newcommand{\norm}[1]{\left\| #1 \right\|}
\newcommand{\absv}[1]{\left| #1 \right|}
\newcommand{\gen}[1]{\left\langle #1 \right\rangle}

\newcommand{\CLF}{\mathrm{CLF}}
\newcommand{\SL}{\mathrm{SL}}

\DeclareMathOperator{\HP}{HP}
\DeclareMathOperator{\HC}{HC}
\DeclareMathOperator{\CC}{CC}

\DeclareMathOperator{\HH}{HH}

\DeclareMathOperator{\KK}{\mathit{KK}}
\DeclareMathOperator{\bop}{\mathbf{b}}
\DeclareMathOperator{\Bop}{\mathbf{B}}

\DeclareMathOperator{\Ad}{Ad}
\DeclareMathOperator{\End}{End}
\DeclareMathOperator{\Hom}{Hom}
\DeclareMathOperator{\Ext}{Ext}
\DeclareMathOperator{\Map}{Map}

\DeclareMathOperator{\rk}{rk}
\DeclareMathOperator{\ev}{ev}
\DeclareMathOperator{\Id}{Id}

\DeclareMathOperator{\ord}{ord}

\title{\normalfont\spacedallcaps{Tracing cyclic homology pairings under twisting of graded algebras}} 

\author{Sayan Chakraborty$^*$ \and Makoto Yamashita$^\dagger$} 

\date{September 11, 2017}

\begin{document}


\renewcommand{\sectionmark}[1]{\markright{\spacedlowsmallcaps{#1}}} 
\renewcommand{\subsectionmark}[1]{\markright{\thesubsection~#1}} 
\lehead{\mbox{\llap{\small\thepage\kern10em\color{halfgray} \vline}\color{halfgray}\hspace{0.5em}\rightmark\hfil}} 

\pagestyle{scrheadings} 


\maketitle 


\begin{abstract}
We give a description of cyclic cohomology and its pairing with $K$-groups for $2$-cocycle deformation of algebras graded over discrete groups. The proof relies on a realization of monodromy for the Gauss--Manin connection on periodic cyclic cohomology in terms of the cup product action of group cohomology.
\end{abstract}


{\let\thefootnote\relax\footnotetext{*: \texttt{sayan.c@uni-muenster.de}, Westf\"{a}lische Wilhelms-Universit\"{a}t M\"{u}nster\\
\textdagger: \texttt{yamashita.makoto@ocha.ac.jp}, Ochanomizu University}}


\section*{Introduction}

Cyclic homology theory as initiated by Tsygan~\cite{MR695483} and Connes~\cite{MR823176} provides a natural extension of de Rham theory for manifolds to objects of noncommutative geometry, which are represented by noncommutative algebras. The main subject of this paper is the cyclic cohomology for deformation of algebras graded by some discrete group $G$, with deformation parameter given by exponential of a $2$-cocycle on $G$. This deformation scheme has many good analogies with the theory of (formal) deformation quantization of Poisson manifolds, but at the same time gives a concrete example of strict deformation in the sense of Rieffel, meaning that we can specialize the deformation parameter to concrete numbers and represent the deformed algebra on Hilbert spaces. The most fundamental example in this setting is the celebrated noncommutative torus $\T^2_\theta$, which is represented by a twisted group algebra of $\Z^2$, and is also a deformation quantization of $\T^2$ with respect to the translation invariant Poisson structure.

Our goal is to understand: (1) the cyclic cohomology of such algebras as abstract linear spaces, and (2) the pairing with $K$-groups / cyclic homology groups for deformed algebras, in terms of those for the original algebra and the cohomology of $G$. For the case of the group algebra of $G = \Z^2$ (which leads to the noncommutative torus), these were carried out in~\citelist{\cite{rieffel-irr-rot-pres}\cite{MR595412}\cite{MR731772}\cite{MR823176}}, but beyond the commutative (or finite, in which case this deformation is trivial) case, the full understanding is limited to only a handful of cases. A notable case is the crossed product of $\Z^2$ by a finite subgroup of $\SL_2(\Z)$, where relatively comprehensive description of $K$-groups, cyclic cohomology groups, and their pairings for twisted group algebras are known~\citelist{\cite{MR1332894}\cite{MR1758235}\cite{MR2301936}\cite{arXiv:1403.5848}\cite{MR3427637}}. Let us also mention that another interesting work from a geometric viewpoint is carried out by Mathai~\cite{MR2218025} for the standard traces of twisted group algebras for many interesting classes of groups.

Regarding the point (1) above, usual computation of cyclic cohomology starts with the computation of Hochschild cohomology groups, which serve as analogue of the spaces of differential forms. Once this is done, the cyclic cohomology groups can be obtained using the map induced by Rinehart--Connes operator, playing the role of the exterior derivative. However, while the end result of computation, particularly the periodic cyclic cohomology groups tend to be invariant under deformation of algebra structures, the intermediate Hochschild cohomology is very sensitive to deformation. This is already apparent in the case of noncommutative torus, and there is little hint for the more general case of deformation of graded algebras by $2$-cocycles, even just for the twisted group algebras.

Our approach, instead, is based on the Gauss--Manin connection for periodic cyclic (co)ho\-mo\-lo\-gy introduced by Getzler~\cite{MR1261901}. The monodromy for this connection, if it makes sense, will give a natural isomorphism between the periodic cyclic cohomology of the deformed algebra and that of the original one. This plays a particularly powerful role in the setting of Poisson deformation quantization, and nicely explains the relation between the Poisson homology of a Poisson manifold and the cyclic homology of the deformed algebra. Motivated by this, our viewpoint could be phrased as follows: the group algebra of $G$ represents a quantum group, which is (up to ``Poincar\'{e} dual'') homotopically modeled on the classifying spaces for centralizers of torsion elements of $G$. A $2$-cocycle $\omega_0$ on $G$ represents an (invariant) ``Poisson bivector'' on this quantum group. A G-grading on an algebra corresponds to an action of this quantum group, by which we can transport the ``Poisson bivector'' and related structures.

A new challenge in the setting of strict deformation is that the Gauss--Manin connection is intrinsically defined on the infinite dimensional space of periodic cyclic (co)chains, and there is no general theorem to integrate it. Developing the second named author's previous work~\cite{arXiv:1207.6687} on ``invariant'' cyclic cocycles, we solve this issue by identifying the action of connection operator up to chain homotopy by the cup product action of group cohomology.

As it turns out, this formalism also fits well with recent developments in the operator algebraic $K$-theory of such algebras~\citelist{\cite{MR2608195}\cite{arXiv:1107.2512}}, where one considers the $C([0, 1])$-algebra whose fiber at $t \in [0, 1]$ is the deformation with respect to the rescaled $2$-cocycle $e^{t \sqrt{-1} \omega_0}$. If $G$ satisfies the Baum--Connes conjecture with coefficients, the evaluation at each $t$ induces isomorphism of $K$-groups, which basically solves the problem (2). On the purely algebraic side, the image of ``$C^\infty$-sections'' in periodic cyclic homology are flat with respect to the Gauss--Manin connection. This allows us to express the pairing of $K$-groups and cyclic cohomology of the deformed algebra in terms of the original data and the action of group cohomology. As an application, we look at the twisted group algebras of crystallographic groups. The $K$-theory of untwisted algebra was studied in detail by L\"{u}ck and his collaborators~\citelist{\cite{MR1803230}\cite{MR3054301}\cite{MR2949238}}, and our result give a description of cyclic cohomology and its pairing with $K$-groups for twisted group C$^*$-algebras.

The paper is organized as follows. In Section 1, we recall the basic terminology and prepare notations. In Section 2, exploiting the comparison of Hochschild cohomology of group algebra and group cohomology~\cite{MR814144}, we construct traces on twisted group algebras which are supported on the conjugacy classes of torsion elements. This leads to isomorphism between the cyclic cohomology groups of general $2$-cocycle deformation of graded algebras over such conjugacy classes, solving the problem (1). Section 3 is the main part of the paper, where we compute the monodromy of the Gauss--Manin connection with respect to the identification established in Section 2. In Section 4, we establish $K$-theoretic comparison between Fr\'{e}chet algebraic models and operator algebraic models, as a basis to compute the pairing of operator algebraic $K$-theory and cyclic cohomology of ``smooth'' algebras. In Section 5, we apply our general theory to twisted algebras of crystallographic groups.

\medskip
Acknowledgements. The authors would like to thank Wolfgang L\"{u}ck and Ryszard Nest for fruitful comments. They also thank the following programmes / grants for their support which enabled collaboration for this paper: Simons - Fundation grant 346300 and the Polish Government MNiSW 2015-2019 matching fund; The Isaac Newton Institute for Mathematical Sciences for the programme \textit{Operator algebras: subfactors and their applications}. This work was supported by: EPSRC grant number EP/K032208/1 and DFG (SFB 878). M.Y. thanks the operator algebra group at University of M\"{u}nster for their hospitality during his stay, and Yasu Kawahigashi for financial support (JSPS KAKENHI Grant Number 15H02056).

\section{Preliminaries}

\subsection{Noncommutative calculus}
\label{sec:nc-calc}

Let us first review the fundamental concepts in cyclic homology theory. Basic references are~\citelist{\cite{MR2052770-Cuntz}\cite{MR2052770-Tsygan}}.

Let $A$ be a unital $\C$-algebra, and put $C_n(A) = A^{\otimes n+1}$. Recall that the \emph{Hochschild differential} $C_n(A) \to C_{n-1}(A)$ is given by
\begin{equation}
\label{eq:hochs-diff-def}
 b(a_0 \otimes \cdots \otimes a_n) = \sum_{i=0}^{n-1} (-1)^i a_0 \otimes \cdots \otimes a_i a_{i+1} \otimes \cdots \otimes a_n + (-1)^n a_n a_0 \otimes a_1 \otimes \cdots \otimes a_{n-1},
\end{equation}
while the \emph{Rinehart--Connes differential} $C_n(A) \to C_{n+1}(A)$ is
\begin{multline*}
B(a_0 \otimes \cdots \otimes a_n) = \sum_{i=0}^n (-1)^{ni} \bigl( 1 \otimes a_{n+1-i} \otimes \cdots \otimes a_n \otimes a_0 \otimes \cdots \otimes a_{n-i} \\
+ (-1)^n a_{n+1-i} \otimes \cdots \otimes a_n \otimes a_0 \otimes \cdots \otimes a_{n-i} \otimes 1 \bigr).
\end{multline*}

The \emph{reduced} Hochschild complex of $A$ is given by $\bar{C}_n(A) = A \otimes (\overline{A})^{\otimes n}$, where $\overline{A} = A / \C$. The kernel of quotient map $C_n(A) \to \bar{C}_n(A)$ is stable under $b$, and is contractible (in the sense of complex). Thus, $(\bar{C}_*, b)$ also computes the Hochschild homology of $A$, and its linear dual complex $\bar{C}^*(A)$ computes  the Hochschild cohomology. On this model $B$ is given by
$$
B\colon \bar{C}_n(A) \to \bar{C}_{n+1}(A), \quad a_0 \otimes \cdots \otimes a_n \to \sum_{i=0}^n (-1)^{n i} 1 \otimes a_i \otimes \cdots \otimes a_n \otimes a_0 \otimes \cdots \otimes a_{i-1}.
$$
We also denote the transpose maps on $\bar{C}^*(A)$ by $b$ and $B$.

The \emph{$A$-valued Hochschild $n$-cochains} are by definition elements of $C^n(A; A) = \Hom_\C(A^{\otimes n}, A)$. The direct sum with degree shift $C^{*-1}(A; A)$ is a graded pre-Lie algebra by partial composition $D \circ D'$, and the product map of $A$ regarded as $\mu \in C^2(A; A)$ gives the Hochschild differential $\delta(D) = [\mu, D]$ by the graded commutator. Given such a cochain $D \in C^m(A; A)$, following~\cite{MR1261901} (but in the convention of~\cite{arXiv:1207.6687}) we consider the operators
\begin{multline*}
\Bop_D(a_0 \otimes \cdots \otimes a_n) \\
= \sum_{\substack{0 \le j \le k \\ \le n-m}} (-1)^{n(j+1) + m(k-j)} 1 \otimes a_{j+1} \otimes \cdots \otimes D(a_{k+1} \otimes \cdots \otimes a_{k+m}) \otimes \cdots \otimes a_n \otimes a_0 \otimes \cdots \otimes a_{j}.
\end{multline*}
and
\begin{equation*}
\bop_{D}(a_0 \otimes \cdots \otimes a_n) = (-1)^{m n} D(a_{n-m + 1} \otimes \cdots \otimes a_n) \otimes a_0 \otimes \cdots \otimes a_{n-m}.
\end{equation*}
acting on $C_*(A)$. Then $\iota_D = \bop_D - \Bop_D$ is called the \emph{interior product} by $D$. Another important operation is the \emph{Lie derivative}
\begin{multline*}
L_{ D } (a_0 \otimes \ldots \otimes a_n) = \sum_{j=0}^{n-m} (-1)^{m (j+1)} a_0\otimes \cdots \otimes D(a_{j+1} \otimes \cdots \otimes a_{j+m}) \otimes \cdots \otimes a_n \\
+ \sum_{j=n-m+1}^n (-1)^{n(n- j)} D(a_{j+1} \otimes \cdots \otimes a_0 \otimes a_1 \otimes \cdots \otimes a_{j+m-n-1}) \otimes a_{j+m-n} \otimes \cdots \otimes a_j.
\end{multline*}
Let us denote by $\bar{C}^n(A; A)$ the space of cochains with $D(a_1 \otimes \cdots \otimes a_n) = 0$ whenever $a_i = 1$ for some $i$.  This is a subcomplex for $\delta$, and $\iota_D$ and $L_D$ are operators on $\bar{C}_*(A)$ when $D$ is in this space. We then have the following \emph{Cartan homotopy formula} relating these two kinds of operators.

\begin{proposition}[\citelist{\cite{MR0154906}\cite{MR1261901}\cite{arXiv:1207.6687}*{Proposition~4}}]
When $D \in \bar{C}^m(A; A)$, we have
$$
[b - B, \iota_D] = L_D - \iota_{\delta(D)},
$$
where the left hand side is the graded commutator in $\End(\bar{C}_*(A))$.
\end{proposition}

Let us recall the formalism of \emph{Gauss--Manin connection} on periodic cyclic (co)homology~\cite{MR1261901}. Suppose $(\mu_t)_{t\in I}$ is a `smooth' family of associative product structures parametrized by $I = [0, 1]$ with common multiplicative unit on $A$. We then consider $\Gamma^\infty(I; A)$ which is, as a linear space, the union of smooth functions from $I$ into finite dimensional subspaces of $A$. The product structure is given pointwise by $\mu_t$, that is $(f * f')(t) = \mu_t(f(t), f'(t))$ for $f, f' \in \Gamma^\infty(I; A)$ (so we require a regularity condition on $\mu_t$ to make this well-defined).

Then $\partial_t$ is a Hochschild $1$-cochain of $\Gamma^\infty(I; A)$ which is not a derivation unless $\mu_t$ is constant in $t$. Note that $\delta(\partial_t) = [\mu, \partial_t]$ is up to sign equal to the time derivative of product structure. Hence the Cartan homotopy formula says
$$
[b - B, \iota_{\partial_t}] = L_{\partial_t} + \iota_{\dot{m}_t}.
$$
as operators on $\bar{C}_*(\Gamma^\infty(I; A))$. Note also that $L_{\partial_t}$ is the natural extension of $\partial_t$ to chains by the Leibniz rule.

\subsection{Standard complex for group cohomology}
\label{sec:std-cplx-grp-cohom}

Let $G$ be a discrete group, and $M$ be a commutative group ($\R$, $\C$, $\T = \R/\Z$, or $\C^\times$ for our purposes). The group cohomology $H^*(G; M)$ is equal to the cohomology of the \emph{standard complex} $(C^*(G; M), d)$, where $C^n(G; M) = \Map(G^{n}, M)$ and $d\colon C^n(G; M) \to C^{n+1}(G; M)$ is given by
\begin{multline}
\label{eq:std-cplx-diff}
d \phi(g_1, \ldots, g_{n+1}) = \phi(g_2, \ldots, g_{n+1}) \\
+ \sum_{i=1}^n (-1)^{i} \phi(g_1, \ldots, g_i g_{i+1}, \ldots, g_{n+1}) + (-1)^{n+1} \phi(g_1, \ldots, g_n).
\end{multline}
We formally put $C^0(G; M) = \Map(*, M) = M$ and $d$ to be the zero map from $C^0(G; M)$ to $C^1(G; M)$. Consequently, $H^0(G; M) = M$ and the $2$-cocycles are $2$-point functions satisfying
$$
\phi(h, k)  + \phi(g, h k) = \phi(g, h) + \phi(g h, k).
$$

If $G$ is a finite group, one has $H^n(G; \C) = 0$ for $n > 0$. This is a consequence of the identification $H^*(G;\C) \simeq \Ext_G^*(\C, \C)$ and the semisimplicity of $G$-modules, but more concretely we can use Shapiro's lemma type argument: for example, given any $2$-cocycle $\omega_0$, one defines $\psi \in C^1(G; \C)$ by
$$
\psi(g) = \frac{1}{\absv{G}} \sum_{h\in G} \omega_0(h, g).
$$
Then one can verify $d \psi(g, h) = \omega_0(g, h)$ by concrete computation using the $2$-cocycle identity.

Obviously, $C^*(G) = C^*(G; \C)$ is the dual complex of $C_*(G) = (\C[G^n])_{n=0}^\infty$, endowed with differential $d\colon C_{n+1}(G) \to C_n(G)$ analogous to that of $C^*(G)$.

A group cochain $\phi \in C^n(G)$ is \emph{normalized} if 
$$
\phi(g_1, \ldots, g_i, e, g_{i+1}, \ldots, g_{n-1}) = 0 \quad (0 \le i \le n-1)
$$
and $\phi(g_1, \ldots, g_n) = 0$ whenever $g_1\cdots g_n = e$. (Note that this requirement is stronger than the usual one found for example in~\cite{MR1324339}.)

\begin{proposition}[\citelist{\cite{MR732839}\cite{MR780077}*{Section~4}}]
\label{prop:group-cocycle-normalization}
When $K$ is a field of characteristic zero and $n > 0$, any $K$-valued $n$-cocycle is cohomologous to a normalized one.
\end{proposition}

While the above does not immediately apply to $\T$- or $\C^\times$-valued cocycles, a direct argument still shows $2$-cocycles with these coefficients can still be normalized. Let us also note the following cyclicity of normalized cocycles.

\begin{proposition}
\label{prop:cycl-normaliz-coc}
Let $\phi$ be a normalized $n$-cocycle on $G$. Then we have
$$
\phi((g_1 \cdots g_n)^{-1}, g_1, \ldots, g_{n-1}) = (-1)^n \phi(g_1, \ldots, g_n).
$$
\end{proposition}

\begin{proof}
Compute the value of $0 = d \phi$ on $((g_1 \cdots g_n)^{-1}, g_1, \ldots, g_n)$ and use the normalization condition.
\end{proof}

\subsection{Cyclic cocycles on group algebra}
\label{sec:cyc-coc-grp-alg}

Let us now describe the cyclic theory for group algebra $\C[G]$, which is essentially from~\cite{MR814144} on cyclic homology. Its adaptation to cyclic cohomologies is already explained in~\cite{MR1246282}, but we will later need explicit constructions in this section.

For each $x \in G$, let us denote the centralizer of $x$ in $G$ by $C_G(x)$, and the conjugacy class of $x$ by $\Ad_G(x)$. 

Let $C_n^x(\C[G])$ be the subspace of $C_n(\C[G]) = \C[G]^{\otimes n+1}$ spanned by those $g_0 \otimes \cdots \otimes g_n$ such that $g_0\cdots g_n \in \Ad_G(x)$. The spaces $(C_n^x(\C[G]))_{n=0}^\infty$ give a subcomplex for $b$ and $B$, which leads to the direct sum decomposition
\begin{align*}
\HH_*(\C[G]) &= \bigoplus_{x \in G/_{\Ad} G} H_*(C_*^x(\C[G]), b),&
\HC_*(\C[G]) &= \bigoplus_{x \in G/_{\Ad} G} H^*(\CC(C_*^x(\C[G])), b-B).
\end{align*}
Dualizing this, we also obtain the direct product decomposition
\begin{align*}
\HH^*(\C[G]) &= \prod_{x \in G/_{\Ad} G} H^*(C_*^x(\C[G])', b),&
\HC^*(\C[G]) &= \prod_{x \in G/_{\Ad} G} H^*(\CC(C_*^x(\C[G]))', b-B),
\end{align*}
and so on. Let us denote the factors labeled by $x$ by $\HH^*(\C[G])_x$ and  $\HC^*(\C[G])_x$. We want to describe them in terms of group cohomology.

From now on let us fix $x$, and also choose and fix $g^y = g_x^y \in G$ such that $(g^y)^{-1} x g^{y} = y$ for each element $y$ in $\Ad_G(x)$.

Let $(g_0, \ldots, g_n) \in G^{n+1}$ be such that $g_0 \cdots g_n \in \Ad_G(x)$. We put
$$
y_i = g_i \cdots g_n g_0 \cdots g_{i-1}.
$$
Note that these elements are also in $\Ad_G(x)$.

\begin{lemma}
The elements $g^{y_i} g_i (g^{y_{i+1}})^{-1}$ for $0 \le i < n$, and $g^{y_n} g_n (g^{y_0})^{-1}$ are in $C_G(x)$.
\end{lemma}

\begin{proof}
This follows from direct calculation using $x g^{y} = g^y y$ and $y_i g_i = g_i y_{i+1}$.
\end{proof}

We next consider the map
$$
\Xi\colon C_n^x(\C[G]) \to C_n(C_G(x)), \quad g_0 \otimes \cdots \otimes g_n \mapsto (g^{y_1} g_1 (g^{y_2})^{-1}, \ldots, g^{y_n} g_n (g^{y_0})^{-1}).
$$

\begin{proposition}
\label{prop:Xi-intertwines-b-and-d}
The map $\Xi$ intertwines the Hochschild differential $b$ on $(C_n^x(\C[G]))_{n=0}^\infty$ and $d$ on $C_*(C_G(x))$.
\end{proposition}

\begin{proof}
We can do direct comparison for each terms in~\eqref{eq:hochs-diff-def} and in~\eqref{eq:std-cplx-diff}.
\end{proof}

Consequently, any cocycle on $C_G(x)$ induces a Hochschild cocycle on $\C[G]$ which is supported on the conjugacy class of $x$. For example, the trivial class represented by $1 \in \C = C^0(C_G(x))$ corresponds to the trace $\tau^x(f) = \delta_{\Ad_G(x)}(f) = \sum_{g \in \Ad_G(x)} f(g)$ on $\C[G]$.

In fact, the above proposition corresponds to the fact that $\Xi$ is induced by a map of \emph{cyclic sets} from $\Gamma_*(G, x)$ to $B_*(C_G(x), x)$ in the notation of~\cite{MR1600246}. A cyclic set is a simplicial set
$$
((X_n)_{n=0}^\infty; (d^n_i\colon X_n \to X_{n-1})_{i=0}^n, (s^n_i\colon X_n \to X_{n+1})_{i=0}^n
$$
endowed with $t_n \colon X_n \to X_n$ satisfying certain compatibility conditions~\cite{MR1600246}*{Chapter 6}. On the one hand, $\Gamma_n(G, x)$ is the natural bases of $C^x_n(\C[G])$. On the other, the simplicial set $B_n(G) = G^n$ is a cyclic set by $$
t(g_1, \ldots, g_n) = ((g_1\cdots g_n)^{-1}, g_1, \cdots, g_{n-1}).
$$
The cyclic set $B_*(C_G(x), x)$ has the same underlying simplicial structure as $B_*(C_G(x))$, but the cyclic structure $t$ is given by $(x (g_1\cdots g_n)^{-1}, g_1, \cdots, g_{n-1})$ instead.

\begin{proposition}
\label{prop:hochs-dir-summand-grp-cohom}
The map $\Xi$ induces an isomorphism $\HH^*(\C[G])_x \simeq H^*(C_G(x);\C)$.
\end{proposition}

\begin{proof}
This corresponds to the first half of~\cite{MR814144}*{Theorem I'}. More concretely, consider a map $\Upsilon\colon C_*(C_G(x)) \to C^x_*(\C[G])$ defined by
$$
\Upsilon(g_1, \ldots, g_n) =  (g_1 \cdots g_n)^{-1} x \otimes g_1 \otimes \cdots \otimes g_n.
$$
Then, on the one hand, $\Xi \Upsilon$ is equal to the identity map on $C_*(C_G(x))$. On the other, $\Upsilon \Xi$ can be computed as
$$
g_0 \otimes \cdots \otimes g_n \mapsto g^{y_0} g_0 (g^{y_1})^{-1} \otimes g^{y_1} g_1 (g^{y_2})^{-1} \otimes \ldots \otimes g^{y_n} g_n (g^{y_0})^{-1}.
$$
This map is homotopic to the identity on $C^x_*(\C[G])$ as follows~\cite{MR1711612}*{Section~III.2}. Put
\begin{align*}
\theta_0(g_0 \otimes \cdots \otimes g_n) &= g^{y_0} g_0 \otimes g_1 \otimes \cdots \otimes g_n \otimes (g^{y_0})^{-1},\\
\theta_1(g_0 \otimes \cdots \otimes g_n) &= g^{y_0} g_0 \otimes g_1 \otimes \cdots \otimes g_{n-1} \otimes (g^{y_n})^{-1} \otimes g^{y_n} g_n (g^{y_0})^{-1}, \ldots\\
\theta_n(g_0 \otimes \cdots \otimes g_n) &= g^{y_0} g_0 \otimes (g^{y_1})^{-1} \otimes g^{y_1} g_1 (g^{y_2})^{-1}\otimes \cdots \otimes g^{y_n} g_n (g^{y_0})^{-1}.
\end{align*}
Then the alternating sum
\begin{equation}
\label{eq:chain-homotopy}
h_n = \sum_{i=0}^n (-1)^{n+i+1} \theta_i \colon C^x_n(\C[G]) \to C^x_{n+1}(\C[G])
\end{equation}
satisfies $b h_n + h_{n-1} b = \iota - \Upsilon \Xi$.

Consequently, $\Xi$ and $\Upsilon$ induce isomorphism of cohomology between the dual complexes.
\end{proof}

\begin{lemma}
Let $\phi$ be a normalized $n$-cocycle on $C_G(x)$. The Hochschild cocycle $\phi \circ \Xi$ on $\C[G]$ descends to $\bar{C}_n(\C[G])$.
\end{lemma}

\begin{proof}
We need to check
$$
\phi \circ \Xi (h_0 \otimes \cdots \otimes h_i \otimes e \otimes h_{i+1} \otimes \cdots \otimes h_{n-1}) = 0 \quad (0 \le i \le n-1)
$$
whenever $h_0 \cdots h_{n-1} \in \Ad_G(x)$.

Putting $(g_0, \ldots, g_n) = (h_0, \ldots, h_i, e, h_{i+1}, \ldots, h_{n-1})$, one has $y_{i+1} = y_{i+2}$ in the above notation and hence the $(i+1)$-th component of $\Xi(g_0, \ldots, g_n)$ is $e$. By the normalization condition on $\phi$, we have the required vanishing.
\end{proof}

Let $N_G(x)$ be the quotient of $C_G(x)$ by the subgroup $\gen{x}$ generated by $x$.

\begin{proposition}
Let $\phi$ be a normalized $n$-cocycle on $N_G(x)$. The induced cocycle $\tilde{\phi}$ on $C_G(x)$ satisfies $\tilde{\phi} \Xi B = 0$.
\end{proposition}

\begin{proof}
We want to show the vanishing of $\phi(g^{y_j} g_j (g^{y_{j+1}})^{-1}, \ldots, g^{y_{j-1}} g_{j-1} (g^{y_j})^{-1})$ for each $j$. Using the cyclicity of Proposition~\ref{prop:cycl-normaliz-coc}, this is equal to
$$
(-1)^n \phi((g^{y_j} y_j (g^{y_j})^{-1})^{-1}, g^{y_j} g_j (g^{y_{j+1}})^{-1}, \ldots, g^{y_{j-2}} g_{j-2} (g^{y_{j-1}})^{-1}).
$$
Since $(g^{y_j} y_j (g^{y_j})^{-1})^{-1}$ is equal to $x^{-1}$, which becomes $e$ in $N_G(x)$, the normalization condition implies the vanishing of this term.
\end{proof}

It follows that normalized cocycles on $N_G(x)$ define reduced cyclic cocycles on $\C[G]$. We next want to understand if they survive in the periodic cyclic cohomology. When $x$ is of infinite order, one has $\HC^n(\C[G])_x \simeq H^n(N_G(x); \C)$. Moreover, the Gysin maps associated with the fibration of classifying spaces $S^1 = B\gen{x} \to B C_G(x) \to B N_G(x)$ defines a linear map $f\colon H^n(N_G(x); \C) \to H^{n+2}(N_G(x); \C)$ for each $k$, which is identified with the $S$-map $\HC^n(\C[G])_x \to  \HC^{n+2}(\C[G])_x$~\citelist{\cite{MR814144}*{Section IV}\cite{MR1600246}*{Chapter~7}}. It follows that the periodic cyclic cohomology class associated with a normalized $n$-cocycle $\phi$ on $N_G(x)$ only depends on its image in the inductive limit $\varinjlim_k H^{n+2k}(N_G(x);\C)$.

\begin{remark}
\label{rem:gysin-alt-pic}
The above Gysin map can be described as follows. The Hochschild--Serre spectral sequence associated with the extension $\gen{x} \to C_G(x) \to N_G(x)$ has the $E_2$-page $H^p(N_G(x); H^q(\gen{x}; \C))$, which are concentrated on $q=0, 1$. The $E_2$-differential
$$
d_2\colon H^p(N_G(x); H^1(\gen{x}; \C)) \to H^{p+2}(N_G(x); H^0(\gen{x}; \C))
$$
can be identified with maps $H^p(N_G(x); \C) \to H^{p+2}(N_G(x); \C)$ which is exactly the Gysin map. The $d_2$ is also directly described by the cup product with the extension class~\cite{MR0197526}. Namely, the central extension $\gen{x} \to C_G(x) \to N_G(x)$ corresponds to a class in $H^2(N_G(x); \gen{x})$. Identifying $\gen{x}$ with $\Z$ and embedding it in $\C$, we obtain a class $c_{G;x}$ in $H^2(N_G(x); \C)$. The cup product by $c_{G;x}$ is $d_2$. It follows that $\varinjlim_k H^{n+2k}(N_G(x);\C)$ is trivial if and only if $c_{G;x}$ represents a nilpotent element in the group cohomology of $N_G(x)$.
\end{remark}

In particular, if $N_G(x)$ has bounded cohomology degree, the conjugacy class of $x$ does not contribute in $\HP^*(\C[G])$. Let us denote the class of discrete groups such that any subgroup has bounded cohomology degree by $\cC$.

\begin{lemma}
\label{lem:class-of-cohom-nice-groups}
The class $\cC$ contains finite groups, finitely generated commutative groups, free groups, fundamental groups of acyclic manifolds, and is closed under extensions.
\end{lemma}

\begin{proof}
One obtains the statement about extensions from Hochschild--Serre spectral sequence.
\end{proof}

When $x$ is of finite order, the situation is much simpler. First the finiteness of $\langle x \rangle$ implies $H^q( \langle x \rangle; \C) = 0$ for $q > 0$, hence the Hochschild--Serre spectral sequence implies $H^n(C_G(x); \C) \simeq H^n(N_G(x); \C)$. Usually the following proposition is stated in the homological form, but the proof carries over without change since it is about the simplicial structure of the underlying cyclic set $B_*(C_G(x);x)$.

\begin{proposition}[\citelist{\cite{MR814144}*{Section IV}\cite{MR1600246}*{Chapter~7}}]
\label{prop:tors-class-HP-comput}
Suppose that $x \in G$ is of finite order. Then the $S$-map $\HC^n(\C[G])_x \to  \HC^{n+2}(\C[G])_x$ is identified with the canonical inclusion as the direct summand
$$
\bigoplus_{k=0}^{\lfloor\frac{n}{2}\rfloor} H^{n-2k}(N_G(x); \C) \to \bigoplus_{k=0}^{\lfloor\frac{n}{2}\rfloor+1} H^{n-2k+2}(N_G(x); \C),
$$
hence one has
$$
\HP^*(\C[G])_x \simeq \bigoplus_{k \equiv * \bmod{2}} H^k(N_G(x); \C) = \bigoplus_{k \equiv * \bmod{2}} H^k(C_G(x); \C).
$$
\end{proposition}

\subsection{Action of group cohomology on graded algebras}

Let us next review ring and module structure on cyclic cohomology of group algebras and graded algebras, mainly from~\citelist{\cite{MR1246282}\cite{arXiv:1207.6687}}.

Let $A$ and $A'$ be two unital $\C$-algebras. Then the complex $C_*(A \otimes A')$ can be regarded as a product $C_*(A) \times C_*(A')$ of \emph{mixed modules}~\cite{MR1600246}*{Chapter~4}, while $C_*(A) \otimes C_*(A')$ is again a mixed modules by differentials $b \otimes \Id + (-1)^* \Id \otimes b$ and $B \otimes \Id + (-1)^* \Id \otimes B$. There is a map $C_*(A) \otimes C_*(A) \to C_*(A) \times C_*(A')$ called the \emph{shuffle product}, which is an quasi-isomorphism for $b-B)$. When $\phi$ (resp.~$\psi$) is a cyclic $m$-cocycle on $A$ (resp.~cyclic $n$-cocycle on $A'$), the \emph{cup product} $\phi \cup \psi$ is a cyclic $(m+n)$-cocycle on $A \otimes A'$ which corresponds to the cocycle $\phi \otimes \psi$ in the dual complex of $C_*(A) \otimes C_*(A')$.

\begin{remark}
\label{rem:cup-prod-compar}
There is another approach $\phi \shrpbin \psi$, which goes through the correspondence between the cyclic cocycles and closed graded traces on differential graded algebras $(\Omega^*, d\colon \Omega^* \to \Omega^{*+1})$ with homomorphism $A \to \Omega^0$~\cite{MR823176}. Then $\phi \shrpbin \psi$ is given by the graded tensor product of the objects involved. When $\phi$ (resp. $\psi$) is a cyclic $m$-cocycle on $A$ (resp. cyclic $n$-cocycle on $A'$), we have
$$
\frac{1}{(m+n)!} [\phi \shrpbin \psi] = \frac{1}{m! n!}[\phi \cup \psi]
$$
in $\HC^{m+n}(A \otimes A')$.
\end{remark}

Let $A$ be a $G$-graded algebra, that is, $A$ decomposes as a direct sum $\bigoplus_{g \in G} A_g$ satisfying $1 \in A_e$ and $A_g A_h \subset A_{gh}$. This can be encoded as a coaction homomorphism
$$
\alpha \colon A \to A \otimes \C[G], \quad a \mapsto a \otimes g \quad (a \in A_g).
$$
There is also a direct product decomposition
$$
\HP^*(A) = \prod_{x \in G/_{\Ad} G} \HP^*(A)_x,
$$
analogous to the case of group algebra.

We then have an action of $\HP^*(\C[G])$ on $\HP^*(A)$, given by $\phi \lhd \psi = \alpha^*(\phi \cup \psi)$ in the above notation. In particular, with $A = \C[G]$, we obtain a ring structure on $\HP^*(\C[G])$. The trace $\tau^x = \delta_{\Ad_G(x)}$ is an idempotent, and for general $G$-graded $A$, it acts on $\HP^*(A)$ as the projection onto $\HP^*(A)_x$.

\begin{proposition}
Let $x$ be of finite order. The factor $\HP^*(\C[G])_x$ is a subring, whose product structure is isomorphic to the cup product on $H^*(C_G(x);\C)$.
\end{proposition}

\begin{proof}
That $\HP^*(\C[G])_x$ is a subring follows from the fact that $\delta_{\Ad_G(x)}$ is a central idempotent for the above product. It remains to identify the product structure. 

On the one hand, by~\cite{arXiv:1207.6687}*{Corollary~9}, the cup product on $\HC^*(B_*(N_G(x)))$ coincides with that of $H^*(N_G(x))$ under the standard comparison map. On the other, there is a map of cyclic modules from $C^x_*(\C[G])$ to $\C[B_*(N_G(x))]$ by composing $\Xi$ with the natural surjection $\C[B_*(C_G(x),x)] \to \C[B_*(N_G(x))]$. The cup product is natural for map of cyclic modules, hence we obtain the assertion.
\end{proof}

\begin{remark}
The above proposition is, at least to some extent, known to experts. For example~\cite{MR1246282} contains a very similar module structure based on $\phi \shrpbin \psi$. But we feel that some details, such as a precise form of product structure on $\HC^k(\C[G])_x \simeq H^k(N_G(x); \C) \oplus H^{k-2}(N_G(x); \C) \oplus \cdots$, and the compatibility under the periodicity map $S$, need some additional elaboration.
\end{remark}

\begin{proposition}
Let $x$ be an element of infinite order in $G$. The $S$-map $\HC^n(A)_x \to \HC^{n+2}(A)_x$ is equal to the cup product by the extension class $c_{G;x} \in \HC^2(\C[G])_x = H^2(N_G(x); \C)$.
\end{proposition}

\begin{proof}
This is proved in~\cite{MR1246282}*{Theorem 6.2} under a different convention of the $S$-map and cup product (see Remark~\ref{rem:cup-prod-compar}). Let us give an alternate shorter proof: the $S$-map in our convention is the cup product by $v \in \HC^2(\C)$ characterized by $v(1, 1, 1) = -\frac{1}{2}$. Identifying $\HC^n(A)_x \otimes \HC^2(\C)$ with $\HC^n(A)_x \otimes \HC^2(\C) \otimes \HC^0(\C[G])_x$ and using the associativity of cup product, we can identify $S(\phi)$ with $\phi \cup S(1_{H^*(N_G(x))})$ for $\phi \in \HC^n(A)_x$. By Remark\ref{rem:gysin-alt-pic}, we know that $c_{G;x}$ represents $S(1_{H^*(N_G(x))})$.
\end{proof}

\begin{corollary}[cf.~\cite{MR1246282}*{Corollary 6.6}]
If $\HP^*(\C[G])_x$ is trivial, then $\HP^*(A)_x$ is also trivial.
\end{corollary}

\subsection{Deformation of graded algebras}

Let $A$ be a $G$-graded algebra, and $\omega$ be a normalized $\C^\times$-valued $2$-cocycle. We can then define a new associative product on $A$ by
$$
a *_\omega b = \omega(g_1, g_2) a b
$$
for homogeneous elements $a \in A_{g_1}$ and $b \in A_{g_2}$. Thanks to the normalization condition on $\omega$, $1_A$ is still the unit for this product. We denote this algebra by $A_{\omega}$, and its element corresponding to $a \in A$ by $a^{(\omega)}$.

There is an embedding of algebra $A_{\omega} \to A \otimes \C_{\omega}[G]$ given by
\begin{equation}
\label{eq:twist-coaction}
a^{(\omega)} \mapsto a \otimes \lambda^{(\omega)}_g \quad (a \in A_g).
\end{equation}

When $\omega$ is a $\T$-valued normalized $2$-cocycle on $G$, we can represent $\C_\omega[G]$ on $\ell_2(G)$ by bounded operators, as
$$
(\lambda^{(\omega)}_gf)(h) = \omega(g,g^{-1} )f(g^{-1} h) \quad (f \in \ell_2(G)),
$$
called the \emph{$\omega$-twisted left regular representation}. The \emph{reduced twisted group C$^*$-algebra} $C^*_{r,\omega}(G) = C^*_r(G, \omega)$  is defined as the norm closure of $\C_\omega[G]$ inside $B(\ell_2(G))$.

When $A$ is a $G$-C$^*$-algebra, the Banach space
$$
\ell_1(G; A) = \biggl\{ f\colon G \to A \mid \norm{f}_1 = \sum_{g \in G} \norm{f(g)} < \infty \biggr\}
$$
has a structure of Banach $*$-algebra by convolution product. When $\omega$ is a normalized $\T$-valued cocycle, we can again twist the product of $\ell_1(G; A)$ into a Banach $*$-algebra, with respect to the same involution map. If $G$ is amenable, $C^*_{r,\omega}(G)$ is the C$^*$-enveloping algebra of $\ell_1(G) = \ell_1(G; \C)$ with this product.

More generally, if $A$ is the reduced section algebra of a Fell bundle over $G$ (that is, a C$^*$-algebra with injective coaction $A \to A \otimes C^*_r(G)$), we can define a C$^*$-algebraic model of $A_\omega$ on the right Hilbert $A_e$-module $L^2(A)_{A_e}$, so that $A_\omega$ is again the reduced section algebra of another Fell bundle.

\subsection{Fr\'{e}chet algebras}

Finally, \emph{a Fr\'{e}chet algebra} is given by a (possibly nonunital) $\C$-algebra $A$ and a sequence of (semi)norms $\norm{a}_m$ ($a \in A$) for $m = 1, 2, 3, \ldots$, such that $A$ is complete with respect to the locally convex topology defined by the $\norm{a}_m$, and that the product map $A \times A \to A$ is (jointly) continuous. Replacing $\norm{a}_m$ with $\sum_{k\le m} \norm{a}_k$ if necessary, we may and do assume that the seminorms are \emph{increasing} ($\norm{a}_m \le \norm{a}_{m+1}$). An important condition one may ask on these seminorms is the \emph{$m$-convexity}: each $\norm{a}_m$ is submultiplicative, that is $\norm{a b}_m \le \norm{a}_m \norm{b}_m$. If we can characterize the topology of $A$ by seminorms satisfying this, $A$ is then called an \emph{$m$-algebra}.

\section{Inducing traces on twisted algebras}

Let $\omega_0(g,h)$ be a $\C$-valued normalized $2$-cocycle on $G$. We consider the twisted group algebra for $\C^\times$-valued $2$-cocycle $\omega^t(g,h) = e^{t\omega_0(g,h)}$, so that the new product is $g *_t h = \omega^t(g,h) gh$.

We want to `deform' the trace $\tau = \delta_{\Ad_G(x)}$ on $\C[G]$ to a one on $\C_{\omega^t}[G]$. Note that when $g h \in \Ad_G(x)$,
$$
\delta_{\Ad_G(x)}(g *_t h) = \omega^t(g,h), \quad \delta_{\Ad_G(x)}(h *_t g) = \omega^t(h,g),
$$
but we do not necessarily have $\omega^t(g,h) = \omega^t(h,g)$. We want to correct this situation by setting $\tau^x_{\omega^t}(g) = e^{t \xi(g)} \delta_{\Ad_G(x)}(g)$ for some function $\xi$. Then $\tau^x_{\omega^t}$ becomes a trace if we have
\begin{equation}
\label{eq:xi-motivate}
\omega_0(g,h) + \xi(gh) = \omega_0(h,g) + \xi(hg)
\end{equation}
This means that the bilinear extension of $\omega_0(g,h) - \omega_0(h,g)$ is the Hochschild coboundary of the linear extension of $\xi$. Let us put $\omega_0^a(g,h) = \omega_0(g,h) - \omega_0(h,g)$.

\begin{lemma}
The bilinear extension of $\omega_0^a$ represents a Hochschild $1$-cocycle on $\C[G]$.
\end{lemma}

\begin{proof}
 We need to show
 $$
 \omega_0^a(gh, k) - \omega_0^a(g, hk) + \omega_0^a(kg, h) = 0
 $$
 whenever $g h k \in \Ad_G(x)$. Expanding the definition of $\omega_0^a$, this is the same as
 $$
 \omega_0(g h, k) + \omega_0(h k, g) + \omega_0(k g, h) = \omega_0(k, g h) + \omega_0(g, h k) + \omega_0(h, k g).
 $$
 Adding $\omega_0(g, h) + \omega_0(h, k) + \omega_0(k, g)$ to both sides and using the cocycle identity, we indeed obtain the equality.
\end{proof}

By Proposition~\ref{prop:hochs-dir-summand-grp-cohom}, $\omega_0^a$ is a Hochschild coboundary if and only if its image in $H^1(C_G(x); \C)$ is trivial. Since $H^1(C_G(x); \C) = Z^1(C_G(x); \C)$, so $\omega_0^a$ is a coboundary if and only if its pullback by $\Upsilon$ vanishes.

\begin{proposition}
\label{prop:omega-0-a-cocycle}
Suppose $x$ is of finite order. Then the function $\omega_0^a \circ \Upsilon(g) = \omega_0^a(g^{-1} x, g)$ on $C_G(x)$ is trivial.
\end{proposition}

\begin{proof}
Using $g^{-1} x = x g^{-1}$ and that $\omega_0$ is a normalized $2$-cocycle, we have $\omega_0^a(g^{-1} x, g) = - \omega_0^a(x, g^{-1})$. Note that $x$ and $g^{-1}$ play the same role in this expression, and $x \in C_G(g^{-1})$. Thus, $\omega_0^a(h, g^{-1})$ as a function in $h \in C_G(g^{-1})$ is in $Z^1(C_G(g^{-1}); \C) = \Hom(C_G(g^{-1}), \C)$. Since $x$ is of finite order, it has to vanish on $x$.
\end{proof}

It follows that we have a solution for $\xi$ in~\eqref{eq:xi-motivate} when $x$ is of finite order.

\begin{example}
\label{ex:Z2-by-Z3-traces}
Suppose we have $G = \Z^2 \rtimes \Z_3$, where the generator $1$ of $\Z_3$ acts on $\Z^2$ by the matrix  \[
   A=
  \left[ {\begin{array}{cc}
   -1 & -1 \\
   1 & 0 \\
  \end{array} } \right].
\]
In the following we denote elements of $G$ by triples $(a, b, i)$ for $a, b \in \Z$ and $i \in \Z_3$. This group has the following conjugacy classes.
\begin{gather*}
\{(a,b,0) \mid (a, b) \in \Z_3 . z\} \quad (z \in \Z_3 \backslash \Z^2) , \quad \{(a, b, 1) \mid a \equiv b \bmod{3} \}, \quad \{(a, b, 1) \mid a \equiv b+1 \bmod{3} \}, \\ \quad \{(a, b, 1) \mid a \equiv b+2,  \bmod{3}\},  \quad \{(a, b, 2) \mid a \equiv b \bmod{3} \}, \quad \{(a, b, 2) \mid a \equiv b+1 \bmod{3} \}, \\ \quad \{(a, b, 2) \mid a \equiv b+2,  \bmod{3}\},
\end{gather*}
The class of first kind with $z = (0,0)$ contributes by rank $2$ in the $K_0$ group, while the ones for nontrivial $z$ do not contribute. The rest of the classes contribute by rank $1$ for each. Let us consider the $2$-cocycle $\omega_0$ on $\Z^2$ given by $\omega_0((a, b), (c, d)) = a d - b c$, and extend it to $G$ by
$$
\omega_0((a, b, g), (c, d, h)) = \omega_0((a, b), (c, d)^g) \quad (a, b, c, d \in \Z, g, h \in \Z_4).
$$
This represents a generator of $H^2(G; \C)$. We then take $\omega(g,h) = \exp(-\sqrt{-1} \frac{\theta}{2} \omega_0(g, h))$.

Let us consider the case $x = (0,0,1)$. Then we have
\begin{align*}
\omega((a,b,0),(c,d,1)) &= \omega((a,b),(c,d)) = \exp\left(-\frac{\sqrt{-1} \theta}{2} (a d - b c)\right),\\
\omega((c,d,1), (a,b,0)) &= \omega((c,d), (-a-b,a)) = \exp\left(-\frac{\sqrt{-1} \theta}{2} ( ca+ad+bd)\right).
\end{align*}
Thus, in the twisted group algebra $\C_\omega[G]$, we have
\begin{align}
\label{eq:walters-fourier-1}
\lambda^{(\omega)}_{(a,b,0)} \lambda^{(\omega)}_{(c,d,1)} &=  \exp\left(-\frac{\sqrt{-1} \theta}{2} ( a d - b c)\right) \lambda^{(\omega)}_{(a+c,b+d,1)},\\
\label{eq:walters-fourier-2}
\lambda^{(\omega)}_{(c,d,1)} \lambda^{(\omega)}_{(a,b,0)} &=  \exp\left(-\frac{\sqrt{-1} \theta}{2} (ca+ad+bd)\right) \lambda^{(\omega)}_{(c-a-b,d+a,1)}.
\end{align}
In particular, $\delta_{\Ad_G(x)}$ is not a trace, but a correction like
\begin{equation}
\label{eq:walters-fourier-3}
\tau^\omega_{x}(\lambda_{(a,b,1)}) = \exp\left(-\frac{\sqrt{-1} \theta( a^2 + a b + b^2)}{6} \right) \delta_{\Ad_G(x)}(a,b,1)
\end{equation}
is. Let us relate the extra factor appearing here to the above general discussion. On the one hand, using the chain homotopy $h_i$ in~\eqref{eq:chain-homotopy}, $\omega_0^a b = 0$ and $\omega_0^a \Upsilon = 0$ implies that $\xi$ can be taken as $\omega_0^a h_0$. On the other, for $y = (a,b,1) \in \Ad_G(x)$, the element $g^y = (\frac{b-a}3,\frac{-a-2b}3, 0)$ satisfies $(g^y)^{-1} x g^y = y$. Thus $\xi$ is given by 
$$
\xi(g_0) = \omega_0^a h_0 (g_0) = \omega_0(g^{g_0} g_0, (g^{g_0})^{-1}) - \omega_0((g^{g_0})^{-1}, g^{g_0} g_0) = \frac{a^2 + a b + b^2}{3}
$$
for $g_0 = (a, b, 1)$ with $a \equiv b \bmod 3$, and $\xi(g_0) = 0$ otherwise. To check that the functional~\eqref{eq:walters-fourier-3} agrees on~\eqref{eq:walters-fourier-1} and \eqref{eq:walters-fourier-2}, one needs
\begin{multline*}
ad - bc + \frac{1}{3}\left( (a+c)^2 + (a+c)(b+d) + (b+d)^2 \right) =\\
c a + a d + b d + \frac{1}{3}\left( (c-a-b)^2 + (c-a-b)(a+d) + (a+d)^2 \right),
\end{multline*}
which is indeed the case.
\end{example}

\subsection{Twisting cyclic cocycles}

Let us assume that $x$ is of finite order, so that we always have $\xi$ satisfying~\eqref{eq:xi-motivate} and the trace $\tau^x_{\omega^t}$ as above on $\C_{\omega^t}[G]$. Then, if $\tau'$ is a trace on $A$, $\tau' \otimes \tau^x_{\omega^t}$ is a trace on $A \otimes \C_{\omega^t}[G]$.

Let $A$ be $G$-graded, and $(\Omega^*, A \to \Omega^0, \tilde{\tau}\colon \Omega^n \to \C)$ be a closed dg-trace model for some cyclic cocycle on $A$. Suppose that $\Omega^*$ is also graded over $G$, and that the homomorphism respects the grading (this is the case for the universal model $\Omega^*(A)$ given by $\Omega^k(A) = \bar{C}_k(A)$). Then the twisting of $\Omega^*$ by $\omega$ is a subalgebra of $\Omega^* \otimes \C_{\omega^t}[G]$ by the inclusion of the form~\eqref{eq:twist-coaction}. Thus we can restrict the trace $\tilde{\tau} \otimes \tau^x_{\omega^t}$ to $\Omega^*_{\omega^t}$, call it $\tilde{\tau}_{\omega^t}$. Since we have an induced homomorphism $A_{\omega^t} \to \Omega^0_{\omega^t}$, we obtain a dg-algebra trace model
$$
(\Omega^*_{\omega^t}, A_{\omega^t} \to \Omega^0_{\omega^t}, \tilde{\tau}_{\omega^t}\colon \Omega^n_{\omega^t} \to \C).
$$
This way we can define a map (which depends on the choice of $\xi$) from the set of cyclic $n$-cocycles on $A$ to those on $A_{\omega^t}$.

\begin{proposition}
\label{prop:cocycle-deform-formula-1}
Let $x$ be an element of finite order in $G$, and $\phi(a_0, \cdots, a_n)$ be a cyclic cocycle on $A$, supported on the conjugacy class of $x$. If we denote a corresponding closed graded trace corresponding to $\phi$ by $\tilde{\tau}$, the cyclic cocycle on $A_{\omega^t}$ corresponding to $\tilde{\tau} \otimes \tau^x_{\omega^t}$ is given by
$$
\phi^{(t)}(a_0, \ldots, a_n) = \omega^t(g_0, g_1) \cdots \omega^t(g_0 \cdots g_{n-1}, g_n) e^{t \xi(g_0 \cdots g_n)} \phi(a_0, \ldots, a_n)
$$
when $a_i$ are homogeneous elements with $a_i \in A_{g_i}$, $g_0 \cdots g_n \in \Ad_G(x)$.
\end{proposition}

\begin{proof}
The correspondence of $\phi$ and $\tilde\tau$ is given by
$$
\phi(a_0, \ldots, a_n) = \tilde{\tau}(a_0 d a_1 \cdots d a_n).
$$
The $n$-form in $\Omega(A_{\omega^t})$ can be computed as
$$
a_0^{(\omega^t)} d a_1^{(\omega^t)} \cdots d a_n^{(\omega^t)} = \omega^t(g_0, g_1) \cdots \omega^t(g_0 \cdots g_{n-1},g_n) (a_0 d a_1 \cdots d a_n)^{(\omega^t)}
$$
for homogeneous elements $a_i \in A_{g_i}$. Under the coaction by $\C[G]$, the element $(a_0 d a_1 \cdots d a_n)^{(\omega)}$ is mapped to $(a_0 d a_1 \cdots d a_n)^{(\omega)} \otimes g_0 \cdots g_n$, hence $\tilde{\tau} \otimes \tau^x_{\omega^t}$ gives the formula in the right hand side of the assertion.
\end{proof}

The inverse cocycle $\omega^{-t}(g,h) = e^{-t\omega_0(g,h)}$ satisfies $(A_{\omega^t})_{\omega^{-t}} = A$. The associated map of cyclic cocycles for $-\xi$ is the inverse map for the above. It follows that $\phi \mapsto \phi^{(t)}$ is a linear isomorphisms from $\HC^n(A)_x$ to $\HC^n(A_{\omega^t})_x$.

\section{Monodromy of Gauss--Manin connection}
\label{alg-gauss-manin}

Let $\omega_0$ be a $\C$-valued normalized $2$-cocycle on $G$, and $x \in G$ be an element of finite order. Our goal is to compute the monodromy of the Gauss--Manin connection on the spaces $\HP^*(\C_{\omega^t}[G])_x$.

\subsection{Finite centralizer}
\label{sec:fin-centralizer}

Suppose first that the centralizer $C_G(x)$ is finite. Following Section~\ref{sec:std-cplx-grp-cohom}, we put
$$
\psi(g) = \frac{1}{\absv{C_G(x)}} \sum_{h \in C_G(x)} \omega_0(h, g) \quad (g \in C_G(x)),
$$
so that we have $d\psi(g, h) = \omega_0(g, h)$ on $C_G(x) \times C_G(x)$.

In order to analyze traces on the deformed algebras, we want an explicit presentation of $\xi$ on $\Ad_G(x)$ satisfying $\xi(b(g \otimes h)) = \omega_0^a(g, h)$. In the notation of the proof of Proposition~\ref{prop:hochs-dir-summand-grp-cohom},
$$
\omega_0^a h_0(g_0) = \omega_0^a(g^{g_0} g_0, (g^{g_0})^{-1})
$$
is one candidate, as Proposition~\ref{prop:omega-0-a-cocycle} implies $\omega_0^a \Upsilon = 0$. For a reason which will become clear soon, we instead take
\begin{equation}
\label{eq:xi-prec-form}
\xi(g_0) = \begin{cases}
\psi(x) + \omega_0^a(g^{g_0} g_0, (g^{g_0})^{-1}) & (g_0 \in \Ad_G(x))\\
0 & \text{otherwise.}
\end{cases}
\end{equation}
Note that adding scalar to $\xi$ does not change $\xi(b(g \otimes h))$, while $e^{t \xi} \tau^x$ is affected. We then put
\begin{equation}
\label{eq:theta-def}
\theta(g_0, g_1) = \begin{cases}
\psi(g^{y_1} g_1 (g^{y_0})^{-1}) + \omega_0(g_1, (g^{y_0})^{-1}) - \omega_0((g^{y_1})^{-1}, g^{y_1} g_1 (g^{y_0})^{-1}) & (g_0 g_1 \in \Ad_G(x))\\
0 & \text{otherwise}
\end{cases}
\end{equation}
with $y_0 = g_0 g_1$ and $y_1 = g_1 g_0$, as in Section~\ref{sec:cyc-coc-grp-alg}. If we extend $\theta$ to $\C[G]^{\otimes 2} = \C[G^2]$ by linearity, it is a reduced Hochschild cochain by the normalization condition on $\omega_0$.

\begin{lemma}
\label{lem:pullback-theta-by-capital-B}
We have $B \theta = \xi$ as a function on $\C[G]$. 
\end{lemma}

\begin{proof}
As the reduced model gives $B(g) = 1 \otimes g$ for $g \in G$, we have
$$
\theta B (g_0) = \psi(x) + \omega_0(g_0, (g^{g_0})^{-1}) - \omega_0((g^{g_0})^{-1}, x)
$$
for $g_0 \in \Ad_G(x)$. Using $x = (g^{g_0} g_0) (g^{g_0})^{-1}$ and the $2$-cocycle identity, we indeed obtain $\xi(g_0)$.
\end{proof}

Let $\omega_0(g, h) g h$ denote the Hochschild $2$-cochain given by the linear extension of $g \otimes h \mapsto \omega_0(g, h) g h$.

\begin{lemma}
\label{lem:pullback-theta-by-small-b}
We have $b \theta = \bop_{\omega_0(g, h) g h} \tau^x$ as a functional on $\C[G^{3}]$, where
$$
\bop_{\omega_0(g, h) g h} \tau^x(g_0, g_1, g_2) = \omega_0(g_1, g_2) \delta_{\Ad_G(x)}(g_0 g_1 g_2).
$$
\end{lemma}

\begin{proof}
Let us keep the notation $y_i$ from the proof of Proposition~\ref{prop:hochs-dir-summand-grp-cohom}. Expanding the definitions, we have
\begin{multline}
\label{eq:eq-1-to-7}
\theta b (g_0 \otimes g_1 \otimes g_2) = \psi(g^{y_2} g_2 (g^{y_0})^{-1}) - \psi(g^{y_1} g_1 g_2 (g^{y_0})^{-1}) + \psi(g^{y_1} y_1 (g^{y_2})^{-1})\\
 + \omega_0(g_2, (g^{y_0})^{-1}) - \omega_0((g^{y_2})^{-1}, g^{y_2} g_2 (g^{y_2})^{-1}) - \omega_0(g_1 g_2, (g^{y_0})^{-1})\\
 + \omega_0((g^{y_1})^{-1}, g^{y_1} g_1 g_2 (g^{y_0})^{-1}) + \omega_0(g_1, (g^{y_2})^{-1}) - \omega_0((g^{y_1})^{-1}, g^{y_1} g_1 (g^{y_2})^{-1}).
\end{multline}
First, (by Proposition~\ref{prop:Xi-intertwines-b-and-d}) the first three terms involving $\psi$ are equal to
\begin{equation}
\label{eq:psi-cobdry}
d \psi(g^{y_1} g_1 (g^{y_2})^{-1}, g^{y_2} g_2 (g^{y_0})^{-1}) = \omega_0(g^{y_1} g_1 (g^{y_2})^{-1}, g^{y_2} g_2 (g^{y_0})^{-1}).
\end{equation}
Combining this with $\omega_0((g^{y_1})^{-1}, g^{y_1} g_1 g_2 (g^{y_0})^{-1})$ in \eqref{eq:eq-1-to-7}, we obtain
\begin{equation}
\label{eq:eq-1-and-5}
\omega_0((g^{y_1})^{-1}, g^{y_1} g_1 (g^{y_2})^{-1}) + \omega_0(g_1 (g^{y_2})^{-1},  g^{y_2} g_2 (g^{y_0})^{-1}).
\end{equation}
Note that the first term of this cancels with the last term of~\eqref{eq:eq-1-to-7}. On the other hand, $\omega_0(g_2, (g^{y_0})^{-1}) - \omega_0((g^{y_2})^{-1}, g^{y_2} g_2 (g^{y_2})^{-1})$ in \eqref{eq:eq-1-to-7} is equal to
\begin{equation}
\label{eq:eq-2-and-3}
\omega_0(g^{y_2} g_2, (g^{y_0})^{-1}) - \omega_0((g^{y_2})^{-1}, g^{y_2} g_2).
\end{equation}
Combining $\omega_0(g_1, (g^{y_2})^{-1})$ in \eqref{eq:eq-1-to-7} and the second term of~\eqref{eq:eq-1-and-5}, we obtain
\begin{equation}
\label{eq:eq-6-and-8}
\omega_0((g^{y_2})^{-1}, g^{y_2} g_2 (g^{y_0})^{-1}) + \omega_0(g_1, g_2 (g^{y_0})^{-1}).
\end{equation}
Then the contribution from~\eqref{eq:eq-2-and-3} and the first term of \eqref{eq:eq-6-and-8} is equal to $\omega_0(g_2, (g^{y_0})^{-1}))$. Combining this with $- \omega_0(g_1 g_2, (g^{y_0})^{-1})$ in \eqref{eq:eq-1-to-7} and the second term of \eqref{eq:eq-6-and-8}, we indeed obtain the factor $\omega_0(g_1, g_2)$ by the cocycle identity.
\end{proof}

Let us consider the family of product structures $\mu_t(g, h) = e^{t \omega_0(g, h)} g h$ on $\C[G]$. We denote by $\omega_0(g, h) g h$ the $2$-cochain on $\C[G]$ characterized by $D(g, h) = \omega_0(g, h) g h$.

\begin{proposition}
\label{prop:gauss-manin-infin-monodromy}
We have $-\xi + \iota_{\omega_0(g, h) g h} \tau^x = (b - B) \theta$.
\end{proposition}

\begin{proof}
Then Lemma~\ref{lem:pullback-theta-by-small-b} says that the Hochschild cochain $D = \partial_t \mu_t$ satisfies $\bop_D \tau^x = b \theta$. By the degree reason we have $\Bop_D \tau^x = 0$, so we also have $\iota_D \tau^x = b \theta$. Combining this with Lemma\ref{lem:pullback-theta-by-capital-B}, we obtain the assertion.
\end{proof}

Let $A$ be a $G$-graded algebra, and let us consider the smooth section algebra $\Gamma^\infty(I; A_{\omega^*})$ as in Section~\ref{sec:nc-calc}, together with the evaluation homomorphisms $\ev_t\colon \Gamma^\infty(I; A_{\omega^*}) \to A_{\omega^t}$.

\begin{theorem}
\label{thm:invar-algebraic}
Let $x$ be an element of $G$ with finite centralizer. Suppose that $A$ is a $G$-graded algebra, $\phi$ is a cyclic cocycle on $A$ supported on the conjugacy class of $x$, and let $\phi^{(t)}$ denote the induced cyclic cocycle on $A_{\omega^t}$ corresponding to $\xi$ in~\eqref{eq:xi-prec-form}. When $c$ is a $(b-B)$-cocycle in $\bar{C}_*(\Gamma^\infty(I;A_{\omega^*}))$, the pairing $\langle \phi^{(t)}, \ev_t c \rangle$ is independent of $t$.
\end{theorem}

\begin{proof}
We want to show that the time derivative of $\langle \phi^{(t)}, \ev_t c \rangle$ is zero. Since the deformations $A \rightsquigarrow A_{\omega^t}$ and $\phi \rightsquigarrow \phi^{(t)}$ are additive in $t$, it is enough to prove this at $t = 0$. Note also that the embedding~\eqref{eq:twist-coaction} can be lifted to
$$
\Gamma^\infty(I; A_{\omega^*}) \to A \otimes \Gamma^\infty(I; \C_{\omega^*}[G]).
$$
Thus, we may replace $A_{\omega^t}$ by $A \otimes \C_{\omega^t}[G]$ and $\phi^{(t)}$ by $\phi \cup \tau^x_{\omega^t}$. Moreover, up to the shuffle product (which is a quasi-isomorphism) $\phi \cup \tau^x_{\omega^t}$ corresponds to $\phi \otimes \tau^x_{\omega^t}$ on the mixed complex $C_*(A) \otimes C_*(\C_{\omega^t}[G])$. Hence it is enough to show that, if $c$ is a cocycle in $C_*(A) \otimes C_*(\Gamma^\infty(I; \C_{\omega^*}[G]))$, the pairing
$$
\langle (\Id \otimes \ev_t)(c), \phi \otimes \tau^x_{\omega^t} \rangle
$$
has zero derivative at $t = 0$. This amounts to
\begin{equation}
\label{eq:time-der-pairing}
\langle \ev_0((\Id \otimes \partial_t) c), \phi \otimes \tau^x \rangle + \langle \ev_0 c, \phi \otimes \xi \rangle = 0,
\end{equation}
where we identify $\xi$ with a Hochschild $0$-cochain supported on $\Ad_G(x)$.

As explained in Section~\ref{sec:nc-calc}, the Cartan homotopy formula for the cochain $\partial_t$ implies
$$
(\Id \otimes \partial_t) c = -(\Id \otimes \iota_{\omega_0(g,h)gh}) c + (\Id \otimes [b-B, \iota_{\partial_t}]) c.
$$
We have $\Id \otimes [b-B, \iota_{\partial_t}] = [b - B, \Id \otimes \iota_{\partial_t}]$, hence the left hand side of \eqref{eq:time-der-pairing} is equal to
$$
-\langle (\Id \otimes (b-B)) \ev_0 c, \phi \otimes \theta \rangle
$$
by Proposition~\ref{prop:gauss-manin-infin-monodromy}. Since $\phi$ is a $(b-B)$-cocycle, this is equal to
$\langle (b-B) \ev_0 c, \phi \otimes \theta \rangle = 0$.
This shows $\partial_t \langle \phi^{(t)}, \ev_t c \rangle = 0$ at $t=0$ as required.
\end{proof}

\subsection{General case}

Let us now turn to the general case. If $x$ is a torsion element, as remarked before Proposition~\ref{prop:tors-class-HP-comput}, $H^n(N_G(x); \C)$ is naturally isomorphic to $H^n(C_G(x); \C)$. Combined with Proposition~\ref{prop:group-cocycle-normalization}, there is a normalized $2$-cocycle $\tilde{\omega}_0$ on $N_G(x)$ which is cohomologous to $\omega_0$ when pulled back to $C_G(x)$. Let $\psi$ denote a function on $C_G(x)$ satisfying
\begin{equation}
\label{eq:psi-general}
\omega_0(g, h) - \psi(g) - \psi(h) + \psi(g h) = \tilde{\omega}_0(g, h) \quad (g, h \in C_G(x)).
\end{equation}
On $\langle x \rangle$, it can be taken as
$$
\psi(x^k) = \frac{1}{\ord x} \sum_{n=0}^{\ord x - 1} \omega_0(x^k, x^n g),
$$
which is independent of $g \in G$ (use the $2$-cocycle condition for $x^k$, $x^n$, and $g$). Following the previous section, we define $\xi \in C^0(\C[G])_x$ and $\theta \in C^1(\C[G])_x$ by the same formulas~\eqref{eq:xi-prec-form} and~\eqref{eq:theta-def}. Lemma~\ref{lem:pullback-theta-by-capital-B} holds without change, and Lemma~\ref{lem:pullback-theta-by-small-b} becomes the following.

\begin{lemma}
\label{lem:pullback-theta-by-small-b-gen}
We have the equality $b \theta = \bop_{\omega_0(g, h) g h} \tau^x + \tilde{\omega}_0 \circ \Xi$ as functionals on $\C[G^3]$.
\end{lemma}

\begin{proof}
The only difference from the proof of Lemma~\ref{lem:pullback-theta-by-small-b} is~\eqref{eq:psi-cobdry}, which needs to be replaced by
\begin{multline*}
d \psi(g^{y_1} g_1 (g^{y_2})^{-1}, g^{y_2} g_2 (g^{y_0})^{-1}) \\
= \omega_0(g^{y_1} g_1 (g^{y_2})^{-1}, g^{y_2} g_2 (g^{y_0})^{-1}) - \tilde{\omega}_0(g^{y_1} g_1 (g^{y_2})^{-1}, g^{y_2} g_2 (g^{y_0})^{-1}).
\end{multline*}
This gives the extra term $\tilde{\omega}_0 \circ \Xi(g_0, g_1, g_2)$.
\end{proof}

Consequently, the monodromy of the Gauss--Manin connection becomes the following.

\begin{proposition}
\label{prop:gauss-manin-infin-monodromy-gen}
We have $-\xi + \iota_{\omega_0(g,h) g h} +  \tau^x \lhd \omega_0 = (b - B) \theta$.
\end{proposition}

\begin{proof}
The cocycle $\tilde{\omega}_0 \circ \Xi$ represents the cup product action of the cohomology class of $\omega_0$ on $\tau^x$.
\end{proof}

The analogue of Theorem~\ref{thm:invar-algebraic} is the following. Note that the assumption on $\exp(t [\omega_0])$ is automatically satisfied if $H^*(C_G(x); \C)$ is bounded in degree, which holds if $G$ is in the class $\cC$ (see Lemma~\ref{lem:class-of-cohom-nice-groups}).

\begin{theorem}
Let $x$ be an element of finite order in $G$. Suppose that $A$ is a $G$-graded algebra, $\phi$ is a cyclic cocycle on $A$ supported on the conjugacy class of $x$, and let $\phi^{(t)}$ denote the induced cyclic cocycle on $A_{\omega^t}$ constructed with $\xi$ as above. Let us also suppose that $\exp(t [\omega_0])$ makes sense in $H^*(C_G(x); \C)$. When $c$ is a $(b-B)$-cocycle in $\bar{C}_*(\Gamma^\infty(I;A_{\omega^*}))$, the pairing $\langle \phi^{(t)} \lhd \exp(-t [\omega_0]), \ev_t c \rangle$ is independent of $t$.
\end{theorem}

\begin{proof}
The proof is parallel to that of Theorem~\ref{thm:invar-algebraic}, but this time we want to check
$$
\langle \ev_0((\Id \otimes \partial_t) c), \phi \otimes \tau^x \rangle + \langle \ev_0 c, \phi \otimes \xi \rangle - \langle \ev_0 c, \phi \otimes \tau^x \lhd \omega_0 \rangle = 0.
$$
The extra term $\tau^x \cup \omega_0$ exactly corresponds to the one in Proposition~\ref{prop:gauss-manin-infin-monodromy-gen}.
\end{proof}

\begin{corollary}
With the same assumption as above, we have
$$
\langle \phi \lhd \exp(t [\omega_0]), \ev_0 c \rangle = \langle \phi^{(t)}, \ev_t c \rangle.
$$
\end{corollary}

\section{Smooth models}

We want to consider Fr\'{e}chet variants of $\Gamma^\infty(I; \C[G])$ in the previous section. When $\cA$ is an $m$-algebra, we may take the projective tensor product $C^\infty(I) \ptimes \cA$, which is again an $m$-algebra (note that tensor product of submultiplicative seminorms are again submultiplicative). It can be identified with the space $C^\infty(I; \cA)$ of $\cA$-valued C$^\infty$-functions on $I$, or the injective tensor product thanks to the nuclearity of $C^\infty(I)$ as a Fr\'{e}chet space, cf.~\citelist{\cite{MR0225131}*{Chapter~44}\cite{MR1082838}*{Theorem~2.3}}.

\subsection{\texorpdfstring{$\ell_1$}{l1}-Schwartz algebra}
\label{sec:ell-1-sch-alg}

Let $\ell\colon G \to [0, \infty)$ be a length function on $G$; a function satisfying $\ell(e) = 0$ and $\ell(g h) \le \ell(g) + \ell(h)$. Let $\cA$ be a Fr\'{e}chet algebra with seminorms $\norm{a}_m$ and let $\alpha\colon G \curvearrowright \cA$ be an action of $G$ which is \emph{$\ell$-tempered}~\cite{MR1217384}:
$$
\forall m \exists C, k, n \colon \norm{\alpha_g(x)}_m \le C (\ell(g) + 1)^k \norm{x}_n.
$$
We then put
\begin{equation}
\label{S-1-G-A-def}
S_1(G; \cA) = \{ f\colon G \to \cA\mid \forall~  k, m\colon \sum_g (\ell(g) + 1)^k \norm{f(g)}_m < \infty \}.
\end{equation}
The seminorms
$$
\norm{f}_{d,m} = \sum_g (\ell(g) + 1)^d \norm{f}_m
$$
topologize $S_1(G; \cA)$, and since $\norm{f}_m$ is increasing in $m$, the seminorms $\norm{f}_m' = \norm{f}_{m,m}$ also topologize $S_1(G; \cA)$, which is an $m$-algebra~\cite{MR1232986}*{Theorem 3.1.7}. As a Fr\'{e}chet space this is just the projective tensor product of $S_1(G) = S_1(G; \C)$ and $\cA$.
Let us denote the twistings of $\ell_1(G)$ (resp.~$S_1(G)$) by $\ell_1(G; \omega)$ (resp.~$S_1(G; \omega)$).
Note that the traces $\tau^x$ are well-defined over $\ell_1(G)$, hence also over $S_1(G)$.

Let $A$ be a $G$-C$^*$-algebra, and $\ell_1(G; A)$ be the crossed product realized on the projective tensor product $\ell_1(G) \ptimes A$. If there are Dirac-dual Dirac morphisms in $\KK^G$ with $\gamma =1 $ in $\KK^{\Ban}_G(\C, \C)$, the assembly map
$$
K^G_*(\ubE G, A) \to K_*(\ell_1(G; A))
$$
is an isomorphism~\cite{MR1914617}. For good $G$ (including Haagerup~\cite{MR1821144} or hyperbolic~\cite{MR2874956}), the assembly maps
$$
K^G_*(\ubE G, A) \to K_*(G \ltimes_r A)
$$
is also an isomorphism, with $G \ltimes_r A$ being the reduced crossed product C$^*$-algebra. It follows that the natural map $K_*(\ell_1(G; A)) \to K_*(G \ltimes_r A)$ is an isomorphism for such groups. In the rest of the section we assume that this is the case.

Let $\omega_0$ be an $\R$-valued $2$-cocycle on $G$, and consider the associated $\T$-valued cocycles $\omega^t(g, h) = \exp(\sqrt{-1} t \omega_0(g, h))$ for $t \in I$. We then obtain a `continuous' family of twisted group C$^*$-algebras $C^*_{r,\omega^t}(G)$, and the C$^*$-algebra of continuous sections $\Gamma(I; C^*_{r,\omega^*}(G))$ plays a central role in the comparison of their K-groups. By definition $\Gamma(I; C^*_{r,\omega^*}(G))$ is a completion of $\ell_1(G; C(I))$ with respect to the representation on Hilbert $C(I)$-module $\Gamma(I; \ell_2(G))$, twisted by the $\cU(C(I))$-valued $2$-cocycle $\omega^*$ on $G$~\cite{MR2608195}.

Recall that $C^*_{r,\omega}(G)$ is strongly Morita equivalent to $G \ltimes_r K(\ell_2(G))$, where $G$ is acting on $K(\ell_2(G))$ by $(\Ad_{\lambda^{(\overline{\omega})}_g})_{g\in G}$~\cite{MR1002543}. Concretely, the imprimitivity bimodule is given by the Hilbert $C^*_{r,\omega}(G)$-module $C^*_{r,\omega}(G) \otimes \ell_2(G)$, where the left action of $G \ltimes_r K(\ell_2(G))$ is given by $\lambda^{(\omega)}_g \otimes \lambda^{(\overline{\omega})}_g$ for $g \in G$ and $1 \otimes T$ for $T \in K(\ell_2(G))$. The $\ell_1$-version of this is a cycle in $\KK^{\Ban}(\ell_1(G; K(\ell_2(G))), \ell_1(G; \omega))$ implementing the isomorphism $K_*(\ell_1(G; K(\ell_2(G)))) \simeq K_*(\ell_1(G; \omega))$. Collecting these, $\ell_1(G; \omega) \to C^*_{r,\omega}(G)$ induces isomorphisms of $K$-groups~\cite{MR2218025}.

\begin{proposition}
\label{prop:ell-1-vers-eval-surj-in-K}
The evaluation map $K_*(\ell_1(G; C(I))) \to K_*(\ell_1(G; \omega^t))$ is surjective
for each $t$.
\end{proposition}

\begin{proof}
We already know the corresponding statement for C$^*$-algebras, and that $K_*(\ell_1(G; \omega^t))$ is isomorphic to $K_*(C^*_{r,\omega^t}(G))$. We also know that $K_*(\ell_1(G; C(I) \otimes K(\ell_2(G)))) \simeq K_(G \ltimes_r C(I) \otimes K(\ell_2(G)))$. Thus, we just need to verify that the evaluation map of the assertion factors the isomorphism $K_*(\ell_1(G; C(I) \otimes K(\ell_2(G)))) \to K_*(\ell_1(G; \omega^t))$ for each $t$.

If we carry out a similar construction for $\ell_1(G; C(I))$ and $\ell_1(G; C(I) \otimes K(\ell_2(G))$, the Banach space $\ell_1(G; C(I; \ell_2(G)))$ gives a cycle in $\KK^{\Ban}(\ell_1(G; C(I) \otimes K(\ell_2(G))), \ell_1(G; C(I)))$, hence a map
$$
K_*(\ell_1(G; C(I) \otimes K(\ell_2(G)))) \to K_*(\ell_1(G; C(I))).
$$
This is the desired factorization.
\end{proof}

Since the natural (semi)norms of $S_1(G; C^\infty(I))$ are stronger than that of $\ell_1(G; C(I))$, we may regard $S_1(G; C^\infty(I))$ as a subspace of $\Gamma(I; C^*_{r,\omega^*}(G))$. In order to control the time derivatives of the product in this algebra, we need to impose the \emph{polynomial growth} assumption on $\omega_0$. Namely, $\omega_0$ is of polynomial growth if
$$
\absv{\omega_0(g, h)} \le C (1 + \ell(g))^M (1 + \ell(h))^M
$$
holds for some large enough $C$ and $M$.

\begin{proposition}[cf.~\cite{MR1217384}*{Theorem~6.7}]
\label{prop:S-1-G-C-infty-I-spec-subalg-ell-1}
If $\omega_0$ is of polynomial growth, $S_1(G; C^\infty(I))$ is a spectral subalgebra of $\ell_1(G; C(I))$.
\end{proposition}

\begin{proof}
First, $S_1(G; C(I))$ is a strongly spectral invariant subalgebra of $\ell_1(G; C(I))$. On this part the argument for~\cite{MR1217384}*{Theorem~6.7} applies without change, since the twisting by $\omega^t$ does not contribute to the norm of these algebras.

We first need to verify that $S_1(G; C^\infty(I))$ is a subalgebra of $S_1(G; C(I))$. Let $f$ and $f'$ be elements of $S_1(G; C^\infty(I))$. Their product is represented by the function
\begin{equation}
\label{eq:section-alg-prod-formula}
(t, g) \mapsto \sum_{h \in G} \omega^t(h, h^{-1} g) f_t(h) f'_t(h^{-1} g).
\end{equation}
Since $\absv{\omega^t(h, h^{-1} g)} = 1$, at each $t$ we indeed obtain an element of $S_1(G; \omega^t)$. We thus need to show that this function is smooth in $t$. Let us denote the constants controlling the growth of $\omega_0$ as above by $C$ and $M$, and put $m_+ = m(1+M)$.

As an illustration let us first give an estimate for the first order derivative. We have
\begin{multline}
\label{eq:prod-first-ord-deriv}
\partial_t (f *_{\omega^*} f')_t(g) \\
= \sum_h \omega^t(h, h^{-1} g) \left ( \omega_0(h, h^{-1} g) f_t(h) f'_t(h^{-1} g) + (\partial_t  f_t(h)) f'_t(h^{-1} g) + f_t(h) \partial_t f'_t(h^{-1} g)\right ).
\end{multline}
The second term can be estimated as $\norm{\partial_t f *_{\omega^t} f'}_{s,1} \le \norm{\partial_t f}_{s,1} \norm{\partial_t f'}_{s,1}$ as in~\eqref{eq:twist-prod-Sob-estim}, so
$$
g \mapsto \sum_h \omega^t(h, h^{-1} g) (\partial_t f_t(h)) f'_t(h^{-1} g)
$$
is an element of $H_\ell^s(G)$. The third term admits a similar estimate. As for the first term, we need to estimate
$$
\sum_g \left( \sum_h \absv{\omega_0(h, h^{-1} g)} \absv{f_t(h)} \absv{f'_t(h^{-1} g)}\right) (1 + \ell(g))^{k}.
$$
If we put $\tilde{f}_t(g) = f_t(g) (1 + \ell(g))^M$, the above is bounded by
\begin{equation}
\label{eq:laff-estim-1-ver}
C \norm{\absv{\tilde{f}} * \absv{\tilde{f}'}}_{k,1} \le C \norm{\absv{\tilde{f}}}_{k,1}  \norm{\absv{\tilde{f}'}}_{k,1} = C \norm{f}_{k_+,1} \norm{f'}_{k_+,1}.
\end{equation}
This shows that $f *_{\omega^*} f'$ indeed has a bounded first order derivative.

Let us now proceed to the higher order derivatives. For $f \in S_1(G; C^m(I))$, put
$$
\norm{f}_m = \sum_{k=0}^m \frac{1}{k!}\max_{s\in I} \norm{\partial_t^k f_s}_{m},
$$
where on the right hand side we take the norm $\norm{f'}_m = \sum_g (1+\ell(g))^m \absv{f'(g)}$ for $f' = \partial_t^k f_s$. Given $f^1, \ldots, f^n \in S_1(G; C^\infty(I))$, we claim that there is an estimate
\begin{equation}
\label{eq:MRZ-type-estim-1-ver}
\norm{f^1 \cdots f^n}_m \le e^{n C} \norm{f^1}_{m_+} \cdots \norm{f^n}_{m_+},
\end{equation}
which implies that $S_1(G; C^\infty(I))$ is closed under product. For $(g_1, \ldots, g_n) \in G^n$, put
$$
\omega^{(n)}_0(g_1, \ldots, g_n) = \omega_0(g_1, g_2) + \omega_0(g_1 g_2, g_3) + \cdots + \omega_0(g_1 \cdots g_{n-1}, g_n).
$$
Then the product $f^1 \cdots f^n$ in $\Gamma(I; C^*_{r,\omega^*}(G))$ is represented by
$$
h_t(g) = \sum_{g_1 \cdots g_n = g} f^1(g_1) \cdots f^n(g_n) e^{t \omega_0^{(n)}(g_1, \ldots, g_n)}.
$$
Then, $\partial_t^k h_s(g)$ is given by
$$
\sum_{a_1 + \cdots + a_{n+1}} \frac{k!}{a_1! \cdots a_{n+1}!} \partial_t^{a_1} f^1_s(g_1) \cdots \partial_t^{a_n} f^n_s(g_n) \omega_0^{(n)}(g_1, \ldots, g_n)^{a_{n+1}} e^{t \omega_0(g_1, \ldots, g_n)}.
$$
Using
\begin{multline*}
\omega_0^{(n)}(g_1, \ldots, g_n)^{a_{n+1}} = \\
\sum_{b_1 + \cdots + b_{n-1}} \frac{a_{n+1}!}{b_1! \cdots b_{n-1}!} \omega_0(g_1, g_2)^{b_1} \omega_0(g_1 g_2, g_3)^{b_2} \cdots \omega_0(g_1 \cdots g_{n-1}, g_n)^{b_{n-1}},
\end{multline*}
we can bound $\frac{1}{k!} \max_s \norm{\partial_t^k h_s}_{m}$ by
\begin{multline*}
\sum_{\substack{a_1 + \cdots + a_n+ b_1  \\ + \cdots + b_{n-1} = k}} \frac{C^{b_1 + \cdots + b_{n-1}}}{a_1! \cdots a_n! b_1! \cdots b_{n-1}!}\\
\times \norm{\left( \cdots\left(\widetilde{\partial_t^{a_1}f^1_s}^{(b_1)} \widetilde{\partial_t^{a_2}f^2_s}^{(b_1)} \right)^{\sim(b_2)} \widetilde{\partial_t^{a_3} f^3}^{(b_2)} \cdots \right)^{\sim(b_{n-1})} \widetilde{\partial_t^{a_n}f^{n}_s}^{(b_{n-1})}}_{m},
\end{multline*}
where we put $\tilde{f}^{(b)}(g) = f(g) (1 + \ell(g))^{M b}$. Repeatedly using the estimates of the form~\eqref{eq:laff-estim-1-ver}, we can bound this by
\begin{multline}
\label{eq:precise-estim-power-sob-norm-1-ver}
\sum_{\substack{a_1 + \cdots + a_n+ b_1  \\ + \cdots + b_{n-1} = k}} \frac{C^{b_1 + \cdots + b_{n-1}}}{a_1! \cdots a_n! b_1! \cdots b_{n-1}!}
\norm{\partial_t^{a_1}f^1_s}_{m+(b_1+\cdots+b_{n-1}) M}\\
\times\norm{\partial_t^{a_2}f^2_s}_{m+(b_1+\cdots+b_{n-1}) M}
\cdots
\norm{\partial_t^{a_n}f^n_s}_{m+b_{n-1} M}
\end{multline}
Then, using $b_1 + \cdots b_{n-1} \le k \le m$ and $\sum_i C^{b_i} / b_i! \le e^C$, we obtain the desired estimate~\eqref{eq:MRZ-type-estim-1-ver}.

It remains to show that $S_1(G; C^\infty(I))$ is closed under taking multiplicative inverses.  Now, suppose that $f \in S_1(G; C^\infty(I))$ is invertible in $S_1(G; C(I))$. The assertion follows if we can show that $f^{-1}$ belongs to $S_1(G; C^m(I))$ for all $m$. If $f \in S_1(G; C^m(I))$ is invertible in $S_1(G; C(I))$, we fix $C_1 < 1$ and approximate $f^{-1}$ by $f' \in S_1(G; C^m(I))$ satisfying
$$
\norm{f' - f^{-1}}_{m_+,0} < \frac{C_1}{e^{3 C} \norm{f}_{m_+,0}},
$$
so that $x = 1 - f' f$ satisfies $\norm{x}_{m_+,0} < C_1 e^{-C}$. Then, for $n > m$, the above estimate for $f^1 = \cdots = f^n = x$ shows
$$
\norm{x^n}_{m,m} \le e^{n C} n^m \norm{x}_{m_+,m}^m \norm{x}_{m_+,0}^{n-m} = o(C_2^n)
$$
for any $C_1 < C_2 < 1$. This shows $f' f$ is invertible in $S_1(G; C^m(I))$.
\end{proof}

\begin{corollary}
The evaluation map $K_*(S_1(G; C^\infty(I))) \to K_*(S_1(G; \omega^t))$ is surjective
for each $t$.
\end{corollary}

\begin{proof}
A similar (but simpler) argument as the above proposition shows that $S_1(G)$ is a spectral subalgebra of $\ell_1(G)$ with respect to the product $*_{\omega^t}$. We thus know that the horizontal arrows in
$$
\xymatrix{
K_*(S_1(G; C^\infty(I))) \ar[d] \ar[r] & K_*(\ell_1(G; C(I))) \ar[d]\\
K_*(S_1(G; \omega^t)) \ar[r] & K_*(\ell_1(G; \omega^t)) 
}
$$
are isomorphisms. Moreover, the right vertical arrow is surjective by Proposition~\ref{prop:ell-1-vers-eval-surj-in-K}.
\end{proof}

\medskip
\begin{remark}
When $G$ is of polynomial growth, there is possibly a more direct approach based on the following result of Schweitzer~\cite{MR1217384}: suppose that $A$ is a $G$-C$^*$-algebra, and $\cA$ be a dense $G$-invariant Fr\'{e}chet subalgebra satisfying the differential seminorm condition such that the induced action is $\ell$-tempered. Then $S_1(G; \cA)$ is a spectral subalgebra of $G \ltimes_r A$. However, this result is hopeless without polynomial growth assumption. Indeed,  $\ell_1(G)$ is not a spectral subalgebra of $C^*_r(G)$ if $G$ contains a free semigroup of two generators~\citelist{\cite{MR0120529}\cite{MR0256185}}.
\end{remark}

\begin{remark}
The subscript $1$ in $S_1(G; A)$ corresponds to the fact that the definition in~\eqref{S-1-G-A-def} is in terms of the `$\ell_1$ norm'. When $A$ is a pre C$^*$-algebra, the $\ell_2$ version makes sense as
$$
S_2(G; A) = \left\{f\colon G \to A \mid \forall ~  k,m \colon \norm{\sum_g \alpha_g^{-1}(f(g)^* f(g)) (1 + \ell(g))^{2 k} }_m < \infty \right\}.
$$
This way $S_2(G; \C)$ is equal to $H_\ell^\infty(G)$, see the next section. However, the spectral invariance is proved only when $G$ is of polynomial growth and $A$ is a commutative C$^*$-algebra~\cite{MR1394381}, and $S_2(G; \ell_\infty G)$ is a spectral subalgebra of $G \ltimes_r \ell_\infty G$ (if and) only if $G$ is of polynomial growth~\cite{MR1957682}.  Note also that we have $S_1(G) = S_2(G)$ if and only if $G$ is of polynomial growth~\cite{MR943303}. For hyperbolic groups, there is a more elaborate choice of rapid decay functions~\cite{MR2647141} which are invariant under holomorphic functional calculus on $C^*_r(G)$ and at the same time supporting many traces.
\end{remark}

\subsection{\texorpdfstring{$\ell_2$}{l2}-Schwartz algebra}

For $s > 0$, put
$$
H_\ell^s(G) = \biggl\{ f\colon G \to \C \mid \sum_{g \in G} \absv{f(g)}^2 (\ell(g) + 1)^{2 s} < \infty \biggr\},
$$
and $H_\ell^\infty(G) = \bigcap_{s > 0} H_\ell^s(G)$. Recall that $G$ is said to be of \emph{rapid decay}~\cite{MR943303} if $H_\ell^\infty(G)$ is a subspace of $C^*_r(G)$ for some proper length function $\ell$. This is equivalent to $H_\ell^s(G) \subset C^*_r(G)$ for large enough $s$. Note that when $G$ is finitely generated, it is rapid decay if and only if $H_\ell^\infty(G) \subset C^*_r(G)$ for the word length $\ell$ for some symmetric generating set. Put
$$
\norm{f}_{\ell,s} = \sqrt{ \sum_{g \in G} \absv{f(g)}^2 (\ell(g) + 1)^{2 s}},
$$
so that $H_\ell^s(G)$ is the completion of $\C[G]$ by this norm. When $H_\ell^s(G) \subset C^*_r(G)$, there is $C > 0$ such that whenever $f$ is supported on $B_R(e)$ (the ball around $e$ of radius $R$), we have $\norm{f}_{C^*_r(G)} \le C R^s \norm{f}_{\ell,s}$, or that there is a polynomial $P(X)$ such that
$$
\norm{f * g}_{C^*_r(G)} \le P(R) \norm{f}_{\ell_2(G)} \norm{g}_{\ell_2(G)}
$$
whenever $f$ is supported on $B_R(e)$. The last condition implies $\norm{f * g}_{\ell,s} \le K(s) \norm{f}_{\ell,s} \norm{g}_{\ell,s}$ for $s \ge \deg P(X)$ and
$$
K(s) = 2^s \sqrt{P(1)^2 + \sum_n \frac{P(2^n)^2}{(1+2^{n-1})^{2s}}},
$$
so that $H_\ell^s(G)$ is a Banach algebra with respect to the rescaled norm $\norm{f}_{\ell,s}' = K(s) \norm{f}_{\ell,s}$~\cite{MR1774859}*{Proposition~1.2}. Modifying the factor $K(s)$ if needed, we may assume that $\norm{f}_{\ell,s}'$ is increasing in $s$. This way $H_\ell^\infty(G)$ is an $m$-algebra. It is also closed under holomorphic functional calculus inside $C^*_r(G)$.

Let $G$ be a rapid decay group, and $\omega$ be a normalized $\T$-valued $2$-cocycle on $G$. As explained in~\cite{MR2218025}*{Proposition~6.11}, $H_\ell^\infty(G)$ is also a subalgebra of $C^*_{r,\omega}(G)$ thanks to the estimate
\begin{equation}
\label{eq:twist-prod-Sob-estim}
\norm{f *_\omega g}_{\ell,s} \le \norm{\absv{f} * \absv{g}}_{\ell,s} \le K(s) \norm{f}_{\ell,s} \norm{g}_{\ell,s}
\end{equation}
for large enough $s$. Let us denote this algebra by $H_\ell^\infty(G; \omega)$. It is again invariant under holomorphic functional calculus in $C^*_{r,\omega}(G)$.

The smooth section algebra $S_1(G; C^\infty(I))$ needs to be replaced by $C^\infty(I; H_\ell^\infty(G))$. Since $C^\infty(I)$ is a nuclear Fr\'{e}chet space, $C^\infty(I; H_\ell^\infty(G))$ which is naturally the injective tensor product of $C^\infty(I)$ and $H_\ell^\infty(G)$ is also the projective tensor product.

Then, an analogue of the first half of Proposition~\ref{prop:S-1-G-C-infty-I-spec-subalg-ell-1} holds.

\begin{proposition}
\label{prop:poly-gro-cocycle-smooth-subalg}\label{prop:poly-gro-cocycle-m-alg}
Suppose that $\omega_0$ is of polynomial growth. Then $C^\infty(I; H^\infty_\ell(G))$ is an $m$-convex subalgebra of $\Gamma(I; C^*_{r,\omega^*}(G))$.
\end{proposition}

\begin{proof}
Let $f$ and $f'$ be elements of $C^\infty(I; H^\infty_\ell(G))$. Their product in $\Gamma(I; C^*_{r,\omega^*}(G))$ is represented by \eqref{eq:section-alg-prod-formula}, and by the estimate \eqref{eq:twist-prod-Sob-estim}, at each $t$ we indeed obtain an element of $H^\infty_\ell(G; \omega^t)$. We thus need to show that this function is smooth in $t$.

The proof is completely analogous to the first half of Proposition~\ref{prop:S-1-G-C-infty-I-spec-subalg-ell-1}, with slight modifications using~\eqref{eq:twist-prod-Sob-estim}. For example, in~\eqref{eq:prod-first-ord-deriv} the second term can be estimated as
$$
\norm{\partial_t f *_{\omega^t} f'}_{\ell,s} \le K(s) \norm{\partial_t f}_{\ell,s} \norm{\partial_t f'}_{\ell,s}
$$
as in~\eqref{eq:twist-prod-Sob-estim}, and the third term admits a similar estimate. As for the first term, we need to estimate
$$
\sum_g \left( \sum_h \absv{\omega_0(h, h^{-1} g)} \absv{f_t(h)} \absv{f'_t(h^{-1} g)}\right)^2 (1 + \ell(g))^{2 s}.
$$
If we put $\tilde{f}_t(g) = f_t(g) (1 + \ell(g))^M$, the above is bounded by
\begin{equation}
\label{eq:laff-estim}
C \norm{\absv{\tilde{f}} * \absv{\tilde{f}'}}_{\ell,s}^2 \le C K(s)^2 \norm{\absv{\tilde{f}}}_{\ell,s}  \norm{\absv{\tilde{f}'}}_{\ell,s} = C K(s)^2 \norm{f}_{\ell,s+\frac{M}{2}}^2 \norm{f'}_{\ell,s+\frac{M}{2}}^2.
\end{equation}
This shows that $f *_{\omega^*} f'$ indeed has a bounded first order derivative.

Let us now proceed to the higher order derivatives. For $f \in C^\infty(I; H^\infty_\ell(G))$, again put
$$
\norm{f}_m = \sum_{k=0}^m \frac{1}{k!}\max_{s\in I} \norm{\partial_t^k f_s}_{\ell,m}.
$$
Let $C$ and $M$ be the constants controlling the growth of $\omega_0$. When $f^1, \ldots, f^n$ in $C^\infty(I; H^\infty_\ell(G))$, instead of \eqref{eq:MRZ-type-estim-1-ver} we claim
\begin{equation}
\label{eq:MRZ-type-estim}
\norm{f^1 \cdots f^n}_m \le \left(K(m_+) e^{C}\right)^n \norm{f^1}_{m_+} \cdots \norm{f^n}_{m_+}
\end{equation}
with $m_+ = m (1 + \frac{M}{2})$ this time. For $(g_1, \ldots, g_n) \in G^n$, put
$$
\omega^{(n)}_0(g_1, \ldots, g_n) = \omega_0(g_1, g_2) + \omega_0(g_1 g_2, g_3) + \cdots + \omega_0(g_1 \cdots g_{n-1}, g_n).
$$
Then, with the product $h = f^1 \cdots f^n$ in $\Gamma(I; C^*_{r,\omega^*}(G))$, we can bound $\frac{1}{k!} \max_s \norm{\partial_t^k h_s}_{\ell,m}$ by
\begin{multline}
\label{eq:precise-estim-power-sob-norm}
\sum_{\substack{a_1 + \cdots + a_n+ b_1  \\ + \cdots + b_{n-1} = k}} \frac{K(m+(b_1+\cdots+b_{n-1}) \frac{M}{2})^n C^{b_1 + \cdots + b_{n-1}}}{a_1! \cdots a_n! b_1! \cdots b_{n-1}!}
\norm{\partial_t^{a_1}f^1_s}_{\ell,m+(b_1+\cdots+b_{n-1}) \frac{M}{2}}\\
\times\norm{\partial_t^{a_2}f^2_s}_{\ell,m+(b_1+\cdots+b_{n-1}) \frac{M}{2}}
\cdots
\norm{\partial_t^{a_n}f^n_s}_{\ell,m+b_{n-1} \frac{M}{2}}
\end{multline}
repeatedly using the estimates of the form~\eqref{eq:laff-estim}. Then, using $b_1 + \cdots b_{n-1} \le k \le m$ and $\sum_i C^{b_i} / b_i! \le e^C$, we obtain the desired estimate~\eqref{eq:MRZ-type-estim}.

The rest proceeds as in the proof of~\cite{MR0144222}*{Lemma~1.2}. Namely, having the $m$-convexity is equivalent to finding a system of convex neighborhoods of $0$  which are idempotent $U \cdot U \subset U$. Putting $B_m(r) = \{f \colon \norm{f}_m < r\} \subset C^\infty(I; H^\infty_\ell(G))$, our claim above implies
$$
B_{m(1+\frac{M}{2})}\left(K(m_+)^{-1} e^{-C}\right)^n \subset B_m(1)
$$
for any $n$. Thus, the convex hull $U_m$ of the $B_{m_+}(K(m_+)^{-1} e^{-C})^n$ for $n = 1, 2, \ldots$ is a neighborhood of $0$ satisfying $U_m^2 = U_m$. Then $\frac{1}{m'} U_m$ for $m, m' = 1, 2, \ldots$ is a fundamental system of idempotent neighborhoods.
\end{proof}

\begin{proposition}
\label{prop:spec-invar-smooth-section-alg}
Under the same assumption as Proposition~\ref{prop:poly-gro-cocycle-smooth-subalg}, $C^\infty(I; H_\ell^\infty(G))$ is a spectral subalgebra of $\Gamma(I; C^*_{r,\omega^*}(G))$.
\end{proposition}

\begin{proof}
First we claim that $C(I, H_\ell^\infty(G))$, or equivalently $C(I, H_\ell^m(G))$ when $m$ is sufficiently large, is a spectral subalgebra of $\Gamma(I; C^*_{r,\omega^*}(G))$. The proof for each fiber~\cite{MR2218025}*{Proposition~6.11} carries on to this case, by replacing the norm of $\ell_2(G)$ by that of the Hilbert $\Gamma(I; C^*_{r,\omega^*}(G))$-module $C(I; \ell_2(G))$. (Otherwise we can appeal to~\cite{MR3032813}.)

The rest of proof is analogous to the second half of Proposition~\ref{prop:S-1-G-C-infty-I-spec-subalg-ell-1}. Suppose that $f \in C^\infty(I; H_\ell^\infty(G))$ is invertible in $C(I, H_\ell^\infty(G))$. The assertion follows if we can show that $f^{-1}$ belongs to $C^m(I; H_\ell^m(G))$ when $m$ is sufficiently large. Let us take an element $f' \in C^m(I, H_\ell^m(G))$ satisfying
$$
\norm{f^{-1} - f'}_{C(I, H_\ell^{m_+}(G))} < C_1 \left(K\left(m_+ \right) e^C \norm{f}_{C(I, H_\ell^{m_+}(G))} \right)^{-1}
$$
for some $0 < C_1 < 1$. Thus, we have $\norm{x}_{C(I, H_\ell^{m_+}(G))} < C_1 K\left(m_+\right)^{-1} e^{-C}$ for $x = 1 - f' f$. Using this we are going to show that $\norm{x^n}_m = o(C_2^n)$ for any $C_1 < C_2 < 1$, which implies that $(1-x)^{-1}$, hence $f^{-1} = (1-x)^{-1} f'$ also, is in $C^m(I, H_\ell^m(G))$.

Replacing $f^1, \ldots, f^n$ in the proof of Proposition~\ref{prop:poly-gro-cocycle-m-alg} by $x$, from~\eqref{eq:precise-estim-power-sob-norm} we can bound $\frac{1}{k!} \displaystyle\max_{s\in I}\norm{\partial_t^k x^n_s}_{\ell,m}$ from above by
$$
\sum_{\substack{a_1 + \cdots + a_n \le k}} \frac{K(m_+)^n e^{nC}}{a_1! \cdots a_n!} \norm{\partial_t^{a_1} x_s}_{\ell,m_+} \cdots \norm{\partial_t^{a_n} x_s}_{\ell,m_+}.
$$
Now, let us suppose $n > m$ so that the $a_i$ start to contain $0$. f we fix the possibility of the possible values of $a_1, \ldots, a_n$ up to permutation, the multiplicity in the above  sum is bounded by $n (n-1) \cdots (n-k+1) \le n^m$. Hence we have
$$
\norm{x^n}_m \le K(m_+)^n e^{nC} n^m \norm{x}_{m_+}^m \norm{x}_{C(I, H_\ell^{m_+}(G))}^{n-m}.
$$
By assumption on $\norm{x}_{C(I, H_\ell^{m_+}(G))}$, we indeed have that this is asymptotically bounded by $C_2^n$ as $n \to \infty$.
\end{proof}

\subsection{Effective conjugacy problem}

We now want to use the formula of Proposition~\ref{prop:cocycle-deform-formula-1}. Let $\phi$ be a cyclic cocycle on $A = S_1(G)$ or $H_\ell^\infty(G)$ supported on the conjugacy class of $x$. As follows from the precise form of $\xi$ given in~\eqref{eq:xi-prec-form}, $\phi^{(t)}$ indeed extends to a cyclic cocycle on $A_{\omega^t} = S_1(G; \omega^t)$ or $H_\ell^\infty(G; \omega^t)$ for each $t$.

If the growth of $\xi(g_0 \cdots g_n)$ is bounded by some polynomial in $\ell(g_0), \ldots, \ell(g_n)$, the family $(\phi^{(t)})_t$ is smooth in $t$ and $\theta$ in Section~\ref{alg-gauss-manin} makes sense as a functional on $A \ptimes A$. Since the growth $\omega_0(g, h)$ is bounded by some polynomial in $\ell(g)$ and $\ell(h)$, it is enough to know that $\ell(g^{g_0})$ is bounded by some polynomial in $\ell(g_0)$ when $g_0$ is conjugate to $x$. Note also that if $C_G(x)$ is finite, the growth rate does not depend on the choice of $g^{g_0}$.

Such problem (usually without finiteness assumption on the order of $x$) is known as the \emph{effective conjugacy problem}: consider the conjugacy length function
$$
\CLF_G(n) = \max_{\substack{\text{$x$, $y$ conjugate} \\\ell(x) + \ell(y) = n}} \min_{g x g^{-1} = y} \ell(g).
$$
This is known to be at most linear for the hyperbolic groups~\cite{MR1018749}, for the solvable groups of the form $\Z^k \rtimes_A \Z$ with $A \in \SL_k(\Z)$ contained in a split torus~\cite{MR3449958}, or for the torsion free $2$-step nilpotent groups~\cite{MR2575390}, as well as the crystallographic groups (see the next section). A quadratic bound is known for fundamental groups of prime $3$-manifolds~\cite{MR3421592}, see also~\cite{MR3449958}*{Theorem~5.5}.

Recall that, in the first half of the next theorem, $G$ being `good' means that it satisfies the Baum--Connes conjecture with coefficients and that the natural map $K_*(\ell_1(G; A)) \to K_*(G \ltimes_r A)$ is isomorphism for any $G$-C$^*$-algebra $A$, which are satisfied by the Haagerup groups and the hyperbolic groups.

\begin{theorem}
\label{thm:paring-gen-smooth}
Let $G$ be a good group with polynomially
bounded conjugacy length function, and $\omega_0$ be an $\R$-valued $2$-cocycle of polynomial growth. Furthermore, let $x$ be a torsion element of $G$ such that $e^{t [\omega_0]}$ makes sense in $H^*(C_G(x); \C)$. Then for any cyclic cocycle $\phi$ on $S_1(G)$ supported on the conjugacy class of $x$ and periodic cyclic cycle $c \in \HP_*(S_1(G; C^\infty(I)))$, the pairing
$$
\langle \ev_t c, \phi^{(t)} \cup e^{t [\omega_0]} \rangle
$$
is constant in $t$.

An analogous statement holds for cyclic cocycles on $H_\ell^\infty(G)$ when $G$ is of rapid decay.
\end{theorem}

We thus have $\langle \ev_t c, \phi^{(t)} \rangle = \langle \ev_0 c, \phi \cup e^{-t [\omega_0]} \rangle$ in the above setting. Note also that, as we have seen in the previous section, the $K$-groups of $S_1(G; C^\infty(I))$ and $C^*_{r,\omega^t}(G)$ are naturally isomorphic, so this gives a concrete description of the pairing between $\HC^*(S_1(G, \omega^t))$ and $K_*(C^*_{r,\omega^t}(G))$ in terms of the one between $\HC^*(S_1(G))$ and $K_*(C^*_{r}(G))$.

\section{K-theory of twisted crystallographic group algebras}

Let us consider a discrete $G$ sitting in an exact sequence $1 \to \Z^n \to G \to F \to 1$ such that $F$  is finite, and its induced action on $\Z^n$ is free away from origin. These groups are of polynomial growth, hence are amenable and there is no distinction between full and reduced twisted groups algebras. We also know that various classes of rapid decay functions coincide~\cite{MR943303}, so we just write $S(G)$. By~\cite{sale-thesis}*{Corollary 2.3.18} the conjugacy length function of $G$ has a linear growth. Moreover, one may take $\R^n$ as a model of $\underline{E}G$ (universal space of the proper $G$-actions)~\cite{MR1029389}, hence $G \backslash \R^n$ as a one for $\underline{B}G = G\backslash\underline{E}G$. Any $\R$-valued $2$-cocycle is up to cohomology represented by an $F$-invariant $2$-cocycle on $\Z^n$.

\begin{example}
Let $G$ be the group $\Z^n$. For each $n \times n$ real skew-symmetric matrix $\theta$, we can  construct a 2-cocycle on this group by defining $\omega_\theta(x, y) = \exp(\sqrt{-1}\langle \theta x,y \rangle )$. The corresponding twisted group C$^*$-algebra $C^*(G, \omega_\theta)$ is also called the noncommutative torus $C(\T^n_\theta)$. For $n=2$, since $\theta$ depends on only one number as $\langle \theta x, y \rangle = \frac{\theta'}2(x_2 y_1 - x_1 y_2)$, we call that number $\theta$ also.
\end{example}

\begin{example}
Let $F$ be a finite subgroup of $\GL_n(\Z)$ such that each $W \in F$ leaves $\theta$ invariant, i.e., $W^T\theta W = \theta$. Then we can define a 2-cocycle $\omega_\theta '$ on $G = \Z^n \rtimes F$ by $\omega_\theta '((x,g),(y,h)) = \omega_\theta (x,g.y)$ for $x, y \in \Z^n$ and $g, h \in F$. If each $W \in F$ not identity acts free away from origin, we obtain a crystallographic group. Again note that when $n = 2$, any finite subgroup of $\SL_2(\Z)$ will do.
\end{example}

Let $\cM$ denote the set of conjugacy classes of maximal finite subgroups of $G$. In the following theorem, $\tilde{K}_0(\C[P])$ denotes the reduced $K_0$-group, that is, the kernel of the map $K_0(\C[P]) \to K_0(\C)$ induced by the trivial representation. Note also that the map $K_i(C^*(\Z^n)) \to K_i(C^*(G))$ induced by the inclusion $\Z^n \to G$ factors through the space of coinvariants $\Z \otimes_F K_i(C^*(\Z^n))$.

\begin{theorem}[\cite{MR2949238}]
In the even degree, there exists an exact sequence 	
\begin{equation}
\label{eq:DL-exact-seq}
\begin{tikzcd}
0 \arrow{r} & \displaystyle\bigoplus_{P\in \cM} \tilde{K}_0(\C[P]) \arrow{r} &
   K_0 (C^* (G)) \arrow{r} & K_0 (\underline{B}G) \arrow{r} &  0,
\end{tikzcd}
\end{equation}
which splits after inverting $\absv{F}$. The natural map
$$
\Z \otimes_F K_0(C^*(\Z^n)) \oplus \biggl( \bigoplus_{P\in \cM} \tilde{K}_0(\C[P]) \biggr) \to K_0(C^*(G))
$$
also becomes a bijection after inverting $\absv{F}$. In the odd degree, we have $K_1(C^* (G)) \simeq K_1(\underline{B} G)$.
\end{theorem}

Under the above setting, any nontrivial finite subgroup of $G$ has a finite normalizer~\cite{MR1803230}*{Lemma 6.1}. In particular, any nontrivial torsion element has a finite centralizer. Thus, for the twisted group algebras, the pairing with cyclic cocycles supported by the conjugacy class of such an elements is described by Section~\ref{sec:fin-centralizer}.

As for the factor labeled by the identity element,

For the above examples $H^*(G; \C)$ is bounded in degree by $n$. Hence the assumptions of Theorem~\ref{thm:paring-gen-smooth} are satisfied and we can compute the pairings $\langle \ev_t c, \phi^{(t)} \rangle$ for $c \in \HP_*(S_1(G; C^\infty(I)))$. For $x = e$, the pairing reduces to that between $K^*(B \Z^n) \simeq K^*(\T^n)$ and $\HP^*(C^\infty(\T^n)) \simeq \bigoplus_k H_{*+2k}(\T^n;\C)$ with modification by $e^{t [\omega_0]}$ with $[\omega_0] \in H^2(\T^n; \C)$, while for nontrivial torsion $x$, the pairing with $\tau^x_{\omega^t}$ essentially does not see $t$.

for the even degree, on the image of $K_0(C^*(\Z^n))$ it reduces to the pairing between
$$
\Bigl(\bigwedge^{2*} \C^n\Bigr)^F \simeq H^{2*}(\Z^n; \C)^F \simeq H^{2*}(B\Z^n; \C)^F
$$
and $\Z \otimes_F K_0(C^*(\Z^n)) \simeq \Z \otimes_F H_{2*}(B \Z^n; \Z)$. On the image of $\tilde{K}_0(\C[P])$, we just pick up the coefficient of $e \in G$ in the projections representing the $\tilde{K}_0$-classes. For the odd degree it reduces to the natural pairing of $(\bigwedge^{2*+1} \C^n)^F \simeq H^{2*+1}(\Z^n; \C)^F$ and $K_1(\underline{B}G) \simeq K_1(B G)$.

\medskip
If $F = \Z_m$, the kernel of natural map $\Z \otimes_F K_1(C^*(\Z^n)) \to K_1(C^*(G))$ is annihilated by the multiplication by $m$, the $K$-groups of $C^*(G)$ are free commutative groups whose rank can be explicitly described~\cite{MR2949238}. Let us further assume that $p = m$ is prime. Then any non-trivial finite subgroup $P$ of $G$ must be isomorphic to $\Z_p$ via the restriction of the projection map $G \to \Z_p$. Thus any nontrivial finite subgroup of $G$ represents an element of $\cM$. We have the following precise description of the $K$-groups.

\begin{theorem}[\cite{MR3054301}]
The number $n' = n/(p-1)$ is an integer, and $\absv{\cM} = p^{n'}$. The rank of $K$-homology groups of $\underline{B}G$ are given by
\begin{align*}
\rk K_0(\underline{B}G) &= \frac{2^{n} + p-1}{2p} + \frac{(p-1) p^{n'-1}}{2},&
\rk K_1(\underline{B}G) &= \frac{2^{n} + p-1}{2p} - \frac{(p-1) p^{n'-1}}{2}.
\end{align*}
\end{theorem}

We want to explain how projections in $\tilde{K}_0(\C[P])$ contribute to the $K_0 (C^* (G)$ and $K_0(C^*(G, \omega_\theta'))$. Since $\C[P]$ is isomorphic to the algebra $C(\hat{P}) \simeq \C^p$, $K_0(\C[P])$ is the free abelian group of rank $p$. Now suppose $g$ is a generator of $P$. The minimal projections of $\C[P]$ are given by
$$
Q_{j,g} = \frac{1}{p} \sum_{k=0}^{p-1} \exp\left(\sqrt{-1}\frac{2\pi}{p} j k\right) \lambda_{g^k}
$$
for $j=0, \cdots, p-1$, which also represent a basis of $K_0(\C[P])$. Since $Q_{0,g}$ represents the trivial representation, a basis of $\tilde{K}_0(\C[P])$ is given by $Q_{1,g}, \cdots , Q_{p-1,g}$.

By the above theorem, these projections are still linearly independent in $K_0 (C^*(G))$. To get elements of $K_0(C^*(G, \omega_\theta'))$, we need to modify this presentation a bit. Continuing to denote a generator of $P$ by $g$ the unitary $\lambda_{g}$ has order $p$ in $C^*(G)$. But since we modified the product in  $C^*(G, \omega_\theta')$, the unitary $\lambda_{g}^{(\omega_\theta')}$ need not be so. Still, since $\omega_\theta'$ is cohomologically trivial on the finite group $P$ and hence $\C_{\omega_\theta'}[P] \simeq \C[P]$, we can always multiply suitable $z \in \T$ so that order of the unitary  $z\lambda_{g}^{(\omega_\theta')}$ is $p$. Then a similar formula
$$
Q^{(\theta)}_{j,g} = \frac{1}{p} \sum_{k=0}^{p-1} \exp\left(\sqrt{-1}\frac{2\pi}{p} j k\right) z^k \lambda_{g^k}^{(\omega_\theta')}
$$
for $j=1, \cdots, p-1$ will give projections which are elements of $K_0(C^*(G, \omega_\theta'))$.

\begin{remark}
Yashinski~\cite{yashinski-thesis} proposed another approach to compare $\HP_*(S(\Z^n, \omega_\theta))$ for different $\theta$ based on $A_\infty$-iso\-mor\-phism of Cartan--Eilenberg complexes, which is equivariant under the action of $\Z_m$ as above and would provide alternative route to the above result. While this is an interesting viewpoint, our understanding is that there is a gap in his argument, mainly because the cyclic bicomplex is not functorial for $A_\infty$-morphisms.
\end{remark}

\subsection{Crossed product of \texorpdfstring{$\Z^2$}{Z2} by \texorpdfstring{$\Z_3$}{Z3}}

Let us come back to Example~\ref{ex:Z2-by-Z3-traces}. This group $\Z^2 \rtimes \Z_3$ can be presented as
$$
G = \left\langle u,v, w \mid w^3=e, u v=v u, w u w^{-1} =u^{-1}v, w v w^{-1} =u^{-1} \right\rangle.
$$
Up to conjugation, the finite subgroups of $G$ are generated by one of the elements $w$, $u w$ or, $u^{2} w$ all of which are of order $3$.
Since $\omega_\theta'(w,w) = 1$, the first element still has order $3$ in the twisted group algebra, and the corresponding projections $Q^{(\theta)}_{1,w}$ and $Q^{(\theta)}_{2,w}$ give nontrivial classes $K_0(C^*(G, \omega_\theta'))$. On the other hand,
$$
\omega_\theta'(u w, u w) \omega_\theta'(u w u w, u w) = \omega_\theta(u, u^{-1} v) \omega_\theta(v, v^{-1}) = e^{-\sqrt{-1} \frac{\theta}{2}}
$$
shows that $\lambda_{u w}^{(\omega_\theta')}$ is not of order $3$. But an adjustment like  $\exp(\sqrt{-1}\frac{\theta}{6})\lambda_{u w}^{(\omega_\theta')}$ gives the right order. When $x, g \in G$ are torsion elements and the intersection $\Ad_G(x) \cap \langle g \rangle$ is nontrivial, there is a nontrivial pairing between $Q^{(\theta)}_{j,g}$ and $\tau^x_{\omega_\theta'}$. Theorem~\ref{thm:invar-algebraic} says that, due to an intricate way the coefficients in both are modified by $\theta$, the value of $\tau^x_{\omega_\theta'}(Q^{(\theta)}_{j,g})$ remains constant. Let us compare this to \cites{MR2301936}.

For simplicity, let us write $e(t)$ instead of $\exp(\sqrt{-1} t \theta)$. Then as noticed above $e(\frac{1}{6})\lambda_{u w}^{(\omega_\theta')}$ is an unitary of order 3. Similarly one can see that $e(\frac{2}{3})\lambda_{u^2 w}^{(\omega_\theta')}$ is also an unitary of order 3. Then we have the following projections $Q^{(\theta)}_{1,w}, Q^{(\theta)}_{2,w}, Q^{(\theta)}_{1,u w}, Q^{(\theta)}_{2,u w}, Q^{(\theta)}_{1,u^2 w}, Q^{(\theta)}_{2,u^2t}$  defining nontrivial classes in $K_0(C^*(G, \omega_\theta'))$. 

Recall from \cite{MR2301936} that if $\alpha$ defines an action of $\Z_p$ on $A$,  an $\alpha$-invariant functional $\phi$ on $A$ is said to be an \emph{$\alpha$-trace} if 
$$
\phi(xy) = \phi(\alpha(y)x)
$$
holds for any $x, y \in A$. An $\alpha^s$-trace $\phi$ gives rise to a trace $T_{\phi}$ on $\Z_p \ltimes A$ defined by
$$
T_{\phi}(x_0 + x_1 w + \cdots +x_{p-1} w^{p-1}) = \phi(x_{p-s})
$$
for $x_i \in A$ and $w$ being the copy of $1 \in \Z_p$.

Coming back to the example, let us consider the induced action of $\Z_3$ on $S(\Z^2, \omega_\theta)$. From \cite{MR2301936}*{Theorem 3.3}, we have following $\alpha$-traces $\phi^1_{l}$ on $S(\Z^2, \omega_\theta)$:
$$
\phi^1_{l}(\lambda^{(\omega_\theta) m}_{u}\lambda^{(\omega_\theta) n}_{v}) = e\left(\frac{1}{6}((m-n)^2-l^2)\right)\delta_{3\Z}(m-n-l)
$$
for $l = 0, 1, 2$. Similarly we have following $\alpha^2$-traces $\phi^{2}_{l}$ on $S(\Z^2, \omega_\theta)$:
$$
\phi^{2}_{l}(\lambda^{(\omega_\theta) m}_{u}\lambda^{(\omega_\theta) n}_{v}) = e\left(-mn-\frac{1}{6}((m-n)^2-l^2)\right)\delta_{3\Z}(m-n-l)
$$
for $l = 0, 1, 2$. Each $\phi^i_{l}$ then gives the the traces $T^i_{l}=T_{\phi^i_{l}}$ on $\Z_3 \ltimes S(\Z^2, \omega_\theta) = S(G, \omega'_\theta)$. Using $\lambda^{(\omega_\theta) m}_{u}\lambda^{(\omega_\theta) n}_{v} = e(-\frac{m n}2) \lambda^{(\omega_\theta)}_{u^m v^n}$, it is straightforward to check that the trace $\tau^t_{\omega_\theta'}$ of Example~\ref{ex:Z2-by-Z3-traces} is exactly equal to $T^{2}_0$. 

The pairing of the above projections and traces are given by the following table (cf.~\cite{MR2301936}*{Theorem 1.2}), which is independent of $\theta$ as suggested by Theorem~\ref{thm:invar-algebraic}.

\begin{center}
\begin{tabular}{ l | r | r | r | r | r | r }
     & $\tau^{w^2}_{\omega'_\theta} = T^1_0$ & $\tau^{u w^2}_{\omega'_\theta} =  T^1_1$ & $\tau^{u^2 w^2}_{\omega'_\theta} =  T^1_2$ & $\tau^{w}_{\omega'_\theta} = T^{2}_0$ & $\tau^{u w}_{\omega'_\theta} = T^{2}_1$ & $\tau^{u^2 w}_{\omega'_\theta} =  T^{2}_2$\\
\hline
    $Q^{(\theta)}_{1,w}$ & $\frac{1}{3} e^{\frac{4\pi}{3} \sqrt{-1}}$ & $0$ & $0$ & $\frac{1}{3} e^{\frac{2\pi}{3} \sqrt{-1}}$  & $0$ & $0$ \\
    $Q^{(\theta)}_{2,w}$ & $\frac{1}{3} e^{\frac{2\pi}{3} \sqrt{-1}}$ & $0$ & $0$ & $\frac{1}{3} e^{\frac{4\pi}{3} \sqrt{-1}}$ & $0$ & $0$\\
    $Q^{(\theta)}_{1,u w}$ & $0$ & $0$ & $\frac{1}{3} e^{\frac{4\pi}{3} \sqrt{-1}}$ & $0$ & $\frac{1}{3} e^{\frac{2\pi}{3} \sqrt{-1}}$ & $0$ \\
    $Q^{(\theta)}_{2,u w}$ & $0$ & $0$ & $\frac{1}{3} e^{\frac{2\pi}{3} \sqrt{-1}}$ & $0$ & $\frac{1}{3} e^{\frac{4\pi}{3} \sqrt{-1}}$ & $0$ \\
    $Q^{(\theta)}_{1,u^2 w}$ & $0$ & $\frac{1}{3} e^{\frac{4\pi}{3} \sqrt{-1}}$ & $0$ & $0$ & $0$ & $\frac{1}{3} e^{\frac{2\pi}{3} \sqrt{-1}}$ \\
    $Q^{(\theta)}_{2,u^2 w}$ & $0$ & $\frac{1}{3} e^{\frac{2\pi}{3} \sqrt{-1}}$ & $0$ & $0$ & $0$ & $\frac{1}{3} e^{\frac{4\pi}{3} \sqrt{-1}}$ \\
\end{tabular}
\end{center}

\end{document}